\documentclass[a4paper,12pt]{amsart}
\usepackage{amsfonts, amssymb, amsmath}
\usepackage{enumerate}
\usepackage{esint}
\usepackage{latexsym,amssymb}

\newtheorem{thm}{Theorem}[section]

\newtheorem{cor}[thm]{Corollary}
\newtheorem{lem}[thm]{Lemma}
\newtheorem{prop}[thm]{Proposition}
\newtheorem{defn}[thm]{Definition}
\newtheorem{rem}[thm]{Remark}

\newcommand{\na}{\nabla}
\newcommand{\lb}{\Delta}
\newcommand{\et}{e^{-t\lb}}

\newcommand{\f}{\frac}

\newcommand{\si}{\sigma}
\newcommand{\vc}{\infty}
\newcommand{\dx}{d\mu(x)}

\newcommand{\Rie}{\na\lb^{-1/2}}

\newcommand{\LA}{L_kA^{-1/2}}

\newcommand{\VA}{V^{1/2}A^{-1/2}}
\newcommand{\RR}{\mathbb{R}}

\title[Commutators of  singular integrals on Hardy spaces]{ Boundedness of singular integrals and their commutators with BMO functions on Hardy spaces}         
\author{The Anh Bui}
\address{Department of Mathematics, Macquarie University, NSW 2109,
Australia}
\address{Department of Mathematics, University of Pedagogy, HoChiMinh City,
Vietnam}
\email{the.bui@mq.edu.au, bt\_anh80@yahoo.com}
\author{Xuan Thinh Duong}
\address{Department of Mathematics, Macquarie University, NSW 2109,
Australia}

\email{xuan.duong@mq.edu.au}


\begin{document}

\date{}

\maketitle

\begin{abstract}
In this paper, we establish sufficient conditions for a singular integral $T$ to be bounded from
certain Hardy spaces $H^p_L$ to Lebesgue spaces $L^p$, $0< p \le 1$, and for the commutator of $T$
and a BMO function to be weak-type bounded on
Hardy space $H_L^1$. We then show that our sufficient conditions are applicable to  the following cases:
(i) $T$ is the Riesz transform or a
square function associated with the Laplace-Beltrami operator on a doubling Riemannian manifold,
(ii) $T$ is the Riesz transform
associated with the magnetic Schr\"odinger operator on an Euclidean space, and
(iii) $T = g(L) $ is a singular integral operator defined from the holomorphic functional calculus
of an operator $L$ or the spectral multiplier of a non-negative self adjoint operator $L$.
 \end{abstract}

 \tableofcontents

\section{Introduction and statement of main results}

The Calder\'on-Zygmund theory of singular integral operators has been a central part
of Harmonic analysis and has had extensive applications to estimates on regularity of solutions
to partial differential equations.  There are a number of recent works which study singular integral
operators with non-smooth kernels that are beyond the standard Calder\'on-Zygmund theory.
See, for example  \cite{DM}, \cite{CD2}, \cite{BK} and \cite{AM}. This paper is a study in that direction, aiming to
study boundedness of certain singular integral operators and their commutators with BMO functions. Our
results are applicable to large classes of differential and integral operators which include the
Riesz transforms on manifolds, the holomorphic functional calculi and spectral multipliers of
non-negative self adjoint operators such as the Laplace-Beltrami operators on manifolds
and magnetic Schr\"odinger operators on Euclidean spaces.

\medskip

Let us first explain our framework.
Assume that $(X,d,\mu)$ is a metric measure space
endowed with a distance $d$
and a nonnegative Borel doubling
measure $\mu$ on $X$.
Recall that a measure is doubling provided that there exists a
constant $C>0$ such that for all $x\in X$ and for all $r>0$,
\begin{equation}\label{boubling}
V(x,2r)\leq C V(x, r)<\infty,
\end{equation}
\noindent where $B(x,r)=\{y\in X: d(x,y)< r\}$ and  $V(x,r)=\mu(B(x,r)).$
In particular, $X$ is a space of homogeneous type.
A more general definition and further studies of these spaces can be found in
\cite{CW1}.
 Note that the doubling property  implies  the  following
strong homogeneity property,
\begin{equation}\label{doublingproperty1}
V(x, \lambda r )\leq  c\lambda^n V(x, r )
\end{equation}
for some $c, n > 0$ uniformly for all $\lambda \geq 1$ and $x \in
X$. The smallest value of the parameter $n$ in the
right hand side of (\ref{doublingproperty1}) is a measure of the dimension of the space.
There also exist $c$ and $N, 0 \leq N \leq n$, so that
\begin{equation}\label{doublingproperty2}
V(y, r )\leq  c\Big(1+\f{d(x,y)}{r}\Big)^NV(x, r )
\end{equation}
uniformly for all $x, y \in X$ and $r > 0$. Indeed, property
$(\ref{doublingproperty2})$ with $N = n$ is a direct consequence of
the triangle inequality of the metric $d$ and the strong homogeneity
property. \\
To simplify notation, we will often just use $B$ for $B(x_B, r_B)$.
Also given $\lambda > 0$, we will write $\lambda B$ for the
$\lambda$-dilated ball, which is the ball with the same center as
$B$ and with radius $r_{\lambda B} = \lambda r_B$. For each ball
$B\subset X$ we set
$$
S_0(B)=B \ \text{and} \ S_j(B) = 2^jB\backslash 2^{j-1}B \
\text{for} \ j\in \mathbb{N}.
$$
In this paper  we will assume that there exists an operator $L$ defined on $L^2(X)$.
We consider the following conditions:
\smallskip

\noindent
 ${\bf(H 1)}$ $L$ is a non-negative self-adjoint operator on $L^2({X})$;

\smallskip

\noindent
${\bf(H 2)}$ The operator $L$ generates an analytic semigroup
$\{e^{-tL}\}_{t>0}$ which satisfies the Davies-Gaffney estimate.
That is, there exist constants $C$, $c>0$ such that for any open subsets
$U_1,\,U_2\subset X$,

\begin{equation}\label{D-Gestimate}
|\langle e^{-tL}f_1, f_2\rangle|
\leq C\exp\Big(-{{\rm dist}(U_1,U_2)^2\over c\,t}\Big)
\|f_1\|_{L^2(X)}\|f_2\|_{L^2(X)},\quad\forall\,t>0,
\end{equation}

\noindent for every $f_i\in L^2(X)$ with
$\mbox{supp}\,f_i\subset U_i$, $i=1,2$,
where ${\rm dist}(U_1,U_2):=\inf_{x\in U_1, y\in U_2} d(x,y)$.

 \medskip
 \smallskip

\noindent
${\bf(H3)}$ The kernel $p_t(x,y) $ of $e^{-tL}$   satisfies the Gaussian upper bound, i.e.
there exist constants $C$, $c>0$ such that for almost every $x,y \in X$,
\begin{equation}\label{Gaussianub}
|p_t(x,y)|\leq \f{C}{V(x,\sqrt{t})}\exp\Big(-\f{d^2(x,y)}{ct}\Big), \forall t>0.
\end{equation}

\begin{rem}\label{rem00}

It is easy to check that the Gaussian bound ${\bf(H3)}$ implies condition ${\bf(H 2)}$.

We list a number of examples:

 \medskip

(i) It is well known that the Laplace operator $\Delta$ on the Euclidean space $\mathbb R^n$
satisfies ${\bf(H 1)}$  and ${\bf(H3)}$. So do the second order non-negative self-adjoint divergence form
 operators with real bounded measurable coefficients on $\mathbb R^n$. Second order divergence form
 operators with complex bounded measurable coefficients on $\mathbb R^n$ would satisfy ${\bf(H2)}$, and
 satisfy ${\bf (H3)}$ for low dimensions $n$ but might not satisfy ${\bf (H3)}$ for higher dimensions $n$.
 See for example \cite{Da}.

 \medskip

 (ii)  Schr\"odinger operators or magnetic Schr\"odinger operators with real potentials
 satisfy ${\bf(H 1)}$  and ${\bf(H3)}$, see \cite{Si}.

 \medskip

 (iii) Laplace-Beltrami operators
on all complete Riemannian manifolds satisfy ${\bf(H 1)}$ and ${\bf(H 2)}$ but
do not satisfy ${\bf(H3)}$ in general, \cite{Da}.

\medskip

(iv) Laplace type operators acting on vector bundles satisfy ${\bf(H 1)}$ and ${\bf(H 2)}$, see
\cite{S}.

\end{rem}

\bigskip

Our  aim in this paper is to obtain boundedness of certain  singular integral operators with non-smooth kernels
and boundedness of their commutators via estimates
on related function spaces.
  Recently, the theory of Hardy spaces associated with operators was studied by many authors,
see for examples \cite{ADM}, \cite{AMR}, \cite{AR}, \cite{DY2}, \cite{HLMMY}, \cite{HM}, \cite{HMMc},
\cite{DL} and   \cite{Y}. We denote by
 $H^p_L(X), 0<p\leq 1$, the Hardy spaces associated to the operator $L$. For the precise definition, we refer the reader to Section 2.\\

Assume that $T$ is a bounded operator on $L^2(X)$.
There are a number of known sufficient conditions on $T$ or its associated kernel $k(x,y)$ so that
$L^2$ boundedness of $T$  can be extended to other spaces such as Lebesgue space $L^p$, $p\ne 2$,
Hardy spaces, and BMO spaces. See, for example \cite{HLMMY} and \cite{DL} for boundedness of
holomorphic functional calculi of certain generators of analytic semigroups on Hardy spaces.
It is also a natural question to   consider boundedness of the commutator of
a BMO function $b$ and $T$ which is given by
$$
[b,T]f(x):= T((b(x) - b) f)(x)
$$
for all functions $f$ with compact supports. See, for example \cite{St} Chapter 7
and \cite{DY1}.

 In this paper, we establish a sufficient condition on an $L^2$ bounded operator $T$ so that it implies both the following:

 (i) $T$ is bounded  from the Hardy spaces $H^p_L(X)$ to $L^p(X)$, $0<p\leq 1$; and

 (ii) the commutator $[b,T]$ is bounded from $H^1_L(X)$ to $L^{1,\infty}(X)$ under the extra assumption
 that $T$ is of weak type $(1,1)$.

While the boundedness of singular integral operators whose kernels are not smooth
enough to belong to the standard class of
Calder\'on-Zygmund operators and the boundedness of the commutators of BMO functions and
these operators was studied extensively, see for example \cite{DM, BK, ACDH, DY1, AM}
and their references, the boundedness of the commutators of BMO functions and
these operators at
the end-point spaces is much less well known.
 Our main results in this paper include
a sufficient condition so that weak type $(1,1)$ estimate  of $T$ implies certain weak type
boundedness  of its commutators $[b,T]$ and the condition is general enough to be applicable to
a wide range of operators in Sections 4, 5 and 6.
The main result is as follows.

\begin{thm}\label{thm0} Assume that $L$ is an operator which satisfies  ${\bf(H 1)}$ and  ${\bf(H 2)}$.
Let $0 < p \le 1$. Let $a$ denote a $(p,2,m)$-atom in the Hardy space $H^p_L(X)$ associate to the operator $L$.
(See Section 2 for the precise definition).
Assume that $T$ is a bounded operator on $L^2(X)$ so that $Ta$ satisfies the estimate
\begin{equation}\label{cond0}
\Big(\int_{S_j(B)}  |T a|^2
d\mu\Big)^{\frac{1}{2}}\leq C2^{-2jm}V(B)^{\f{1}{2}-\f{1}{p}}
\end{equation}
for any $(p,2,m)$-atom $a$ supported in the ball $B$ and all $j\geq 2$.
Then we have:
\begin{enumerate}[(i)]
\item $T$ is bounded from $H^p_L(X)$ to $L^p(X)$; and
\item in addition, if $T$ is of weak type $(1,1)$ then the commutator $[b,T]$,
where $b$ is a {\rm BMO} function, maps continuously from $H^1_L(X)$ to $L^{1,\infty}(X)$.
\end{enumerate}
\end{thm}

\begin{rem}\label{rem0}

(a) The main result of Theorem \ref{thm0} is the boundedness of the commutator in (ii).
The result in (i) on boundedness of $T$ on Hardy spaces was proved in many situations
and can be considered as folklore.

\medskip

(b) There is no regularity condition on the kernel of $T$, so in general
$T$ is not a standard  Calder\'on-Zygmund singular integral operator
(whose kernel is required to be H\"older continuous or at least to
satisfy the H\"ormander condition).

  \medskip

(c) In Sections 4, 5 and 6, we apply Theorem \ref{thm0} to prove the boundedness of various singular integral
operators and their commutators which  do not belong to the
class of Calder\'on-Zygmund  operators.

  \medskip

(d) It follows from (i) and interpolation (see \cite[Theorem 9.3]{HLMMY}) that $T$ is bounded from $H^p_L(X)$ to $L^p(X)$ for $0<p\leq 2$.
\end{rem}

The paper is organized as follows. In Section 2, we recall the definition of $H^p_L(X)$, the Hardy space
associated to the operator $L$, and some characterizations of $H^p_L(X)$. The proof of Theorem \ref{thm0} is given in
Section 3. In Section 4,  we consider the Riesz transform $T$  on a doubling manifold and use
Theorem \ref {thm0} to obtain some endpoint estimates of commutators of Riesz transforms on manifolds.
In Section 5, we study the boundedness of Riesz transforms of magnetic Schr\"odinger operators and their commutators.
We will show that the Riesz transforms of magnetic Schr\"odinger operators are bounded from
$H^p_A(\mathbb{R}^n)$ to $L^p(\mathbb{R}^n)$ for all $0<p\leq 1$ and the commutators of Riesz transforms and BMO functions
are bounded from $H^1_A(\mathbb{R}^n)$ to $L^{1,\infty}(\mathbb{R}^n)$.
In Section 6, we show boundedness of holomorphic functional calculus and spectral multipliers of an operator
which generates a holomorphic semigroup with suitable kernel bounds.\\

\section{Hardy spaces associated to operators}
The theory of Hardy spaces associated to non-negative self-adjoint operators satisfying the Davies-Gaffney estimate was developed recently by Hofmann et. al. \cite{HLMMY}. Here, we use the definitions and characterizations of Hardy spaces $H^p_L(X)$  in \cite{HLMMY} and \cite{DL}.
\subsection{Hardy spaces $H^p_{L,  S_h}(X)$ for $p\geq 1$ }
Let $L$ be an operator which satisfies  ${\bf(H 1)}$ and  ${\bf(H
2)}$. Set
$$
H^2(X):=\overline{\mathcal{R}(L)}=\{Lu\in L^2(X): u\in \mathcal{D}(L)\}
$$
where $\mathcal{D}(L)$ is the domain of $L$.

It is known that $L^2(X)=\overline{\mathcal{R}(L)}\oplus
\mathcal{N}(L)$, where $\mathcal{R}(L)$ and $\mathcal{N}(L)$ stand
for the range and the kernel of $L$, and the sum is orthogonal.

Consider the following quadratic operators associated to $L$
\begin{equation}
S_{h, K}f(x)=\Big(\int_0^{\infty}\int_{\substack{  d(x,y)<t}}
|(t^2L)^{K}e^{-t^2L} f(y)|^2
\f{d\mu(y)}{V(x,t)}\f{dt}{t}\Big)^{1/2}, \quad x\in X
\end{equation}
where $K$ is a positive integer and $f\in L^2(X)$. We
shall write $S_{h}$ in place of $S_{h, 1}$. For each integer $K\geq
1$ and $1\leq p<\infty$, we now define
\begin{eqnarray*}
D_{K, p} =\Big\{ f\in H^2(X): \ S_{h, K}f\in L^p(X)\Big\}, \ \ \
1\leq p<\infty.
\end{eqnarray*}
\begin{defn}{\label{def2.2}} Let $L$ satisfy $\bf{(H1)}$ and $\bf{(H2)}$. \

(i) For each $1\leq p\leq 2$, the Hardy space $H^p_{L,S_h }(X)$
associated to $L$  is the completion of the space $D_{1, p}$ in the
norm
$$
 \|f\|_{H_{L,S_h }^p(X)}=  \|S_{h}f\|_{L^p(X)}.
$$

(ii) For each $2<p<\infty$, the Hardy space $H^p_{L}(X)$ associated
to $L$ is the completion of the space $D_{K_0, p}$ in the norm
$$
\|f\|_{H_{L,S_h }^p(X)}=  \|S_{h, K_0}f\|_{L^p(X)}, \ \ \ \
K_0=\big[\f{n}{2}\big]+1.
$$
\end{defn}
It can be verified that
 $H^2_{L,S_h }(X)=H^2(X)\subset L^2(X)$ and  the dual space of $H^p_{L,S_h }(X)$ is
$H^{p'}_{L,S_h }(X)$ where $1/p+ 1/p'=1$ (see Proposition 9.4 of
\cite{HLMMY}). If $L$ satisfies $\bf{(H1)}$ and $\bf{(H3)}$, then it was proved in \cite{AMR} that $H^p_{L,S_h }(X)$ and $L^p(X)$ coincide for all $p\in (1,\vc)$.

\medskip

 \subsection{The atomic Hardy spaces $H^p_{L, at, m}(X)$  for $p\leq 1$}
 In what follows, assume that
\begin{eqnarray}
m\in \mathbb{N}\ \ \ \text{and} \ \  m>\f{n(2-p)}{4p}, \label{e2.2}
\end{eqnarray}

\noindent where the parameter $n$ is a constant in (\ref{doublingproperty1}). Let us denote by ${\mathcal D}(T)$
the domain of an operator $T$.

The notion of a  {\it $(p,2, m)$-atom}, $0<p\leq 1$, associated to
operators on spaces $(X,d,\mu)$, is defined as follows .
\begin{defn} \label{def2.3} A function $a\in L^2(X)$ is said to be a
$(p,2, m)$-atom associated to  an operator $L$ if there exist a
function $b\in {\mathcal D}(L^m)$ and a ball $B$ such that

\medskip

{  (i)}\ $a=L^m b$;

\medskip

{  (ii)}\ supp$L^{k}b\subset B, \ k=0, 1, \dots, m$;

\medskip

{  (iii)}\ $||(r_B^2L)^{k}b||_{L^2(X)}\leq
r_B^{2m}V(B)^{\f{1}{2}-\f{1}{p}},\ k=0,1,\dots,m$.

\medskip

\noindent Obviously, in the case $\mu(X)<\infty$ the constant function $[\mu(X)]^{-\f{1}{p}}$ is also considered to be an atom.
\end{defn}

\begin{defn}
Given $0<p\leq 1$ and $m>\f{n(2-p)}{4p}$, we  say that $f=\sum
\lambda_ja_j$ is an atomic $(p,2,m)$-representation if
$\{\lambda_j\}_{j=0}^\infty\in l^p$, each $a_j$ is a $(p,2,m)$-atom,
and the sum converges in $L^2(X)$. Set
$$
\mathbb{H}^p_{L,at,m}(X)=\{f:f \ \text{has an atomic
$(p,2,m)$-representation}\},
$$
with the norm given by
$$
||f||_{\mathbb{H}^p_{L,at,m}(X)}=\inf\{(\sum|\lambda_j|^p)^{1/2}:
f=\sum \lambda_ja_j \ \text{is an atomic $(p,2,m)$-representation}\}.
$$
The space $H^p_{L,at,m}(X)$ is then defined as the completion of
$\mathbb{H}^p_{L,at,m}(X)$ with respect to the quasi-metric $d$
defined by $d(h,g)=||h-g||_{\mathbb{H}^p_{L,at,m}(X)}$ for all
$h,g\in \mathbb{H}^p_{L,at,m}(X)$.
\end{defn}
In this case the mapping $h\rightarrow ||h||_{H^p_{L,at,m}(X)},
0<p<1$ is not a norm and $d(h,g)=||h-g||_{H^p_{L,at,m}(X)}$ is a
quasi-metric. For $p=1$, the mapping $h\rightarrow
||h||_{H^1_{L,at,m}(X)}$ is a norm. A straightforward argument shows
that $H^p_{L,at,m}(X)$ is complete. In particular, $H^1_{L,at,m}(X)$
is a Banach space. A basic
result concerning these spaces is the following proposition.
\begin{prop}
If an operator $L$ satisfies conditions $\bf{(H1)}$ and $\bf{(H2)}$, then for
every $0<p\leq 1$ and for all integers $m\in \mathbb{N}$ with
$m>\f{n(2-p)}{4p}$, the spaces $H^p_{L,at,m}(X)$ coincide and their
norms are equivalent.
\end{prop}
For the proof, we refer to Theorem 5.1 of \cite{HLMMY} for $p=1$, and Section 3 of \cite{DL} for $p<1$.\\

The notion of a $(p,2,m,\epsilon)$-molecule
associated to an operator $L$ will be described as follows.
\begin{defn}
Let $0<p\leq 1$, $\epsilon>0$ and $m\in \mathbb{N}$. We say that a
function $\alpha \in L^2$ is called a $(p,2,m,\epsilon)$-molecule
associated to $L$ if there exist a function $b\in \mathcal{D}(L^m)$ and a ball
$B$ such that
\begin{enumerate}[(i)]
\item $\alpha = L^mb$;
\item For every $k=0,1,\ldots,m$ and $j=0,1,\ldots,$ the following estimate holds
$$
\|(r_B^2L)^kb\|_{L^2(S_j(B))}\leq
r_B^{2m}2^{-j\epsilon}V(2^jB)^{\f{1}{2}-\f{1}{p}}.
$$
\end{enumerate}
\end{defn}
\begin{prop}\label{mol-pro}
Suppose $0<p\leq 1$, $m>\f{n(2-p)}{4p}$ and $\epsilon>0$. If $\alpha$ is a
$(p,2,m,\epsilon)$-molecule associated to $L$, then $\alpha \in
H^p_{L,at,m}(X)$. Moreover, $\| \alpha \|_{H^p_{L,at,m}(X)}$ is  independent of
$\alpha$.
\end{prop}
For the proof, we refer to \cite{HLMMY} for $p=1$, and \cite{DL} for
$p<1$.

\subsection{A characterization of Hardy spaces associated to operators in terms of square functions}

In Section 2.1, we had the definitions of the Hardy spaces $H^p_{L,S_h}(X)$ for
$p\geq 1$. Now consider the case $0<p<1$.  The space
$H^p_{L,S_h}(X)$ is defined as the completion of
$$
\{f\in H^2(X):||S_hf||_{L^p(X)}<\infty\}
$$
under the norm defined by the $L^p$ norm of the square function; i.e.,
$$
||f||_{H^p_{L,S_h}(X)}=||S_hf||_{L^p(X)}, \ 0<p<1.
$$

Then the ``square function" and ``atomic'' $H^p$ spaces are
equivalent, if the parameter $m>\frac{n(2-p)}{4p}$. In fact, we have
the following result.
\begin{prop}
Suppose $0<p\leq 1$ and $m>\frac{n(2-p)}{4p}$. Then we have
$H^p_{L,at,m}(X)=H^p_{L,S_h}(X)$ and their norms are equivalent.
\end{prop}
\emph{Proof:} For the proof, see \cite{DL}.\\

Consequently, as in the next definition, one may write $H^p_{L,at}(X)$
in place of $H^p_{L,at,m}(X)$ when $m>\frac{n(2-p)}{4p}$. Precisely, we
have the following definition.
\begin{defn}
The Hardy space $H^p_{L}(X), 0<p\leq 1$, is the space
$$
H^p_{L}(X):=H^p_{L,S_h}(X):=H^p_{L,at}(X):=H^p_{L,at,m}(X),\
m>\frac{n(2-p)}{4p}.
$$
\end{defn}

\section{Boundedness of singular integral operators and their commutators}
To prove that an operator $T$ is bounded on the Hardy space $H^p_L(X)$ which possesses an atomic
decomposition, it is not enough in general to prove that $Ta$ is uniformly bounded
for all atomic functions $a$. However, if the operator
 $T$ satisfies extra condition such as being $L^2(X)$ bounded
 (or even the weaker condition of weak type $(2,2)$), then
 the uniform boundedness of $Ta$ does imply the boundedness of
$T$ on $H^p_L(X)$. More precisely, we have the following result.

\begin{prop}\label{prop1}
Suppose that $T$ is a linear (resp. nonnegative sublinear) operator which
maps $L^2(X)$ continuously into $L^{2,\vc}(X)$.
If there exists, for $0<p\leq 1$, a
constant $C$ such that
$$
||Ta||_{L^{p}}\leq C
$$
for all $(p,2,m)$-atoms $a \in H^p_{L}(X)$, then $T$ extends to a
bounded linear (resp. sublinear) operator from $H^p_{L}(X)$
to $L^{p}(X)$.
\end{prop}

\begin{proof} The proof of this proposition is standard. For completeness and convenience
of reader, we give a proof here.

Suppose that $f\in H^p_L(X)\cap H^2(X)$ so that we
may write $f=\sum_{j=1}^\vc \lambda_j a_j$ in $L^2$ sense where $a_j$ are
$(p,2,m)$-atoms and $\Big(\sum_{j=1}^\vc|\lambda_j|^p\Big)^{1/p}\approx
||f||_{H^p_{L}(X)}$. It suffices to show that $|Tf|\leq \sum_{j=1}^\vc
|\lambda_j| |Ta_j|$. \\

 If $T$ is a linear operator, then from the
fact that the sum $\sum_{j=1}^\vc \lambda a_j$ converges in $L^2(X)$
and $T$ is of weak type $(2,2)$, we conclude that $Tf=\sum_{j=1}^\vc
\lambda_j Ta_j$. \\

If $T$ is a nonnegative sublinear
operator and $T$ is of weak type $(2,2)$, one has
\begin{equation*}\begin{aligned}
\mu\Big\{x\in X: \Big|Tf-T\Big(\sum_{j=1}^N \lambda_j
a_j\Big)\Big|>t\Big\}&\leq   \f{C}{t^2}\Big|\Big| \sum_{j=N+1}^\vc \lambda_j
a_j\Big|\Big|_{L^2(X)} .
\end{aligned}\end{equation*}
Hence
\begin{equation*}\begin{aligned}
\lim_{N\rightarrow\vc} \mu\Big\{x\in X: \Big|Tf-T\Big(\sum_{j=1}^N \lambda_j
a_j\Big)\Big|>t\Big\}&\leq  C\lim_{N\rightarrow\vc}\f{1}{t^2}\Big|\Big| \sum_{j=N+1}^\vc \lambda_j
a_j\Big|\Big|_{L^2(X)}= 0.
\end{aligned}\end{equation*}

This implies  that $T\Big(\sum_{j=1}^N \lambda_j
a_j\Big) \rightarrow Tf$ a.e. as $N\rightarrow \infty$. Since $T$ is a nonnegative sublinear operator, we have
\begin{equation*}\begin{aligned}
Tf - \sum_{j=1}^{\infty}T(\lambda_j a_j)&= Tf - T\Big(\sum_{j=1}^N \lambda_j
a_j\Big)+ T\Big(\sum_{j=1}^N \lambda_j
a_j\Big)-\sum_{j=1}^{\infty}T(\lambda_j a_j)\\
&\leq Tf - T\Big(\sum_{j=1}^N \lambda_j
a_j\Big) \rightarrow 0, \ \text{as $N\rightarrow \infty$}.
\end{aligned}\end{equation*}
Thus, $Tf\leq \sum_{j=1}^{\infty}T(\lambda_j a_j)$. The proof is complete.
\end{proof}

\emph{Proof of Theorem \ref{thm0}:}
(i) Proposition \ref{prop1}, it suffices to show that for any $(p,2, m)$-atom $a$, for $m>\f{n(2-p)}{4p}$, associated to the ball $B$, we have $||T a||_{L^p}\leq C$.\\
Indeed, we have
\begin{equation*}
\begin{aligned}
\int_X|T a(x)|^p \dx &= \sum_{j=0}^\infty \int_{S_j(B)}|T a(x)|^p \dx \\
&= \sum_{j=0}^\infty K_j.
\end{aligned}
\end{equation*}
By Jensen's and H\"older's inequalities and (\ref{cond0}), one has, for each $j$,
\begin{equation*}
\begin{aligned}
K_j&\leq V(S_j(B))^{1-\f{p}{2}}||Ta||^p_{L^2(S_j(B))}\\
&\leq CV(2^jB)^{1-\f{p}{2}}2^{-2jmp}V(B)^{\f{p}{2}-1}\\
&\leq C2^{j(n-\f{np}{2})-2mp}.
\end{aligned}
\end{equation*}
This together with $m>\f{n(2-p)}{4p}$ gives
$$
\int_X|T a(x)|^p \dx\leq C\sum_{j=0}^\infty 2^{j(n-\f{np}{2})-2mp}\leq C.
$$
 The proof of (i) is complete.\\

(ii) Suppose that $f\in H^1_L(X)\cap H^2(X)$ so that we
may write $f=\sum_{j=1}^\vc \lambda_j a_j$ in $L^2$ sense where $a_j$ are
$(1,2,m)$-atoms associated to balls $B_j$ and $\sum_{j=1}^\vc|\lambda_j|\approx
||f||_{H^1_{L}}$. It
suffices to show that there exists a constant $c>0$ such that
$$
\mu\{x\in X: |[b,T](\sum_{j=1}^\vc \lambda_j a_j)(x)|>\lambda\}\leq \f{c}{\lambda}\|f\|_{H^1_L}||b||_{BMO}
$$

Using the commutator technique as in \cite{Pe}, we can assume that $b\in L^\vc$. Setting $\displaystyle b_B=\f{1}{V(B)}\int_Bbd\mu$, we have
\begin{equation*}\begin{aligned}
\Big|[b,T]f(x)\Big|&=\Big|T((b(x)-b)f)(x)\Big|\\
&\leq \Big|T((b(x)-b)(\sum_{j\geq 0}\lambda_j a_j))(x)\Big|\\
&\leq \Big|T\Big[\sum_{j\geq 0}\lambda_j(b(x)-b_{B_j}) a_j\Big](x)\Big|+\Big|T\Big[\sum_{j\geq 0}\lambda_j(b-b_{B_j}) a_j\Big](x)\Big|\\
\end{aligned}
\end{equation*}
where in the last inequality we use that fact that both series $\sum_{j\geq 0}\lambda_j(b(x)-b_{B_j}) a_j$ and $\sum_{j\geq 0}\lambda_j(b-b_{B_j}) a_j$ converge in $L^1(X)$.

At this stage, by the similar argument as in Proposition \ref{prop1}, we can write
$$
\Big|T\Big[\sum_{j\geq 0}\lambda_j(b(x)-b_{B_j}) a_j\Big](x)\Big|\leq \sum_{j}|\lambda_j|\Big|[b(x)-b_{B_j}]Ta_j(x)\Big|.
$$
This gives
\begin{equation*}\begin{aligned}
\Big|[b,T]f(x)\Big|&=\Big|T((b(x)-b)f)(x)\Big|\\
&\leq\sum_{j}|\lambda_j|\Big|[b(x)-b_{B_j}]Ta_j(x)\Big| + \Big|T(\sum_{j}\lambda_j[b_{B_j}-b]a_j)(x)\Big|.
\end{aligned}
\end{equation*}
Therefore,
\begin{equation*}\begin{aligned}
\lambda\mu\{x\in X: |[b,T]f(x)|>\lambda\}&\leq \lambda\mu\{x\in X:
\sum_{j}|\lambda_j| |[b(x)-b_{B_j}]Ta_j(x)|>\lambda/2\}\\
&~~+\lambda\mu\{x\in X: |T(\sum_{j}\lambda_j[b_{B_j}-b]a_j)(x)|>\lambda/2\}\\
&:=E_1+ E_2.
\end{aligned}\end{equation*}
Let us estimate $E_2$ first. Since $T$ is of weak type $(1,1)$, one
has, by H\"older's inequality
\begin{equation*}\begin{aligned}
E_2&\leq C \sum_{j\geq 0}|\lambda_j|\int_X |[b_{B_j}-b(x)]a_j(x)|\dx\\
&\leq C \sum_{j\geq 0}|\lambda_j| \|(b_{B_j}-b)\|_{L^2(B_j)}\|a_j\|_{L^2(B_j)}\\
&\leq C \sum_{j\geq 0}|\lambda_j|\|b\|_{BMO}V(B_j)^{1/2}V(B_j)^{-1/2}
\leq C\|f\|_{H^1_L}\|b\|_{BMO}.
\end{aligned}\end{equation*}

We now estimate $E_1$. Obviously,
\begin{equation*}\begin{aligned}
E_1&\leq C\sum_{j\geq 0}|\lambda_j|\int_X |[b(x)-b_{B_j}]Ta_j(x)|\dx\\
&= C\sum_{j\geq 0}|\lambda_j|\sum_{k=0}^\vc \int_{S_k(B_j)} |[b(x)-b_{B_j}]Ta_j(x)|\dx\\
&\leq C\sum_{j\geq 0}|\lambda_j|\sum_{k=0}^\vc \int_{S_k(B_j)} |[b(x)-b_{2^kB_j}]Ta_j(x)|\dx\\
&~~~+C\sum_{j\geq 0}|\lambda_j|\sum_{k=0}^\vc \int_{S_k(B_j)} |[b_{B_j}-b_{2^kB_j}]Ta_j(x)|\dx.
\end{aligned}\end{equation*}
By H\"older's inequality, (\ref{cond0}) and the
fact that $|b_{B_j}-b_{2^kB_j}|\leq ck||b||_{BMO}$, we have,
\begin{equation}\label{es1}\begin{aligned}
\int_{S_k(B_j)}
|[b(x)-b_{2^kB_j}]Ta_j(x)|\dx
&\leq C||[b(x)-b_{2^{k}B_j}||_{L^2(2^{k}B_j)}||Ta_j||_{L^2(S_k(B_j))}\\
&\leq C V(2^{k}B_j)^{1/2}||b||_{BMO}2^{-2mk}V(B_j)^{-1/2}\\
&\leq C 2^{k(\f{n}{2}-2m)}||b||_{BMO}.
\end{aligned}\end{equation}
and
\begin{equation}\label{es2}
\begin{aligned}
\int_{S_k(B_j)}
|[b_{2^{k}B_j}-b_{B_j}]Ta_j|\dx
&\leq Cj\|b\|_{BMO}\int_{S_{k}(B_j)}
|Ta_j|\dx\\
&\leq C kV(2^{k}B_j)^{1/2}\|b\|_{BMO}2^{-2mk}V(B_j)^{-1/2}\\
&\leq C k2^{k(\f{n}{2}-2m)}\|b\|_{BMO}.
\end{aligned}
\end{equation}
The estimates (\ref{es1}), (\ref{es2}) together with $m>\f{n}{4}$ imply that
$$E_1\leq C\sum_{j\geq 0}|\lambda_j|\|b\|_{BMO}\approx \|f\|_{H^1_L}\|b\|_{BMO}.$$

It remains to extend $[b, T]$ to the whole space $H^1_L(X)$. For each $f\in H^1_L$, there exists a sequence $(f_n)$ in $H^1_L(X)\cap H^2(X)$ so that $f_n\rightarrow f$ in $H^1_L(X)$. We have, for all $k,n\in \mathbb{N}$,
$$
\Big\| [b, T]f_n - [b, T]f_k\Big\|_{L^{1,\vc}(X)}\leq \Big\| [b, T](f_n - f_k)\Big\|_{L^{1,\vc}(X)}\leq C\|b\|_{BMO}\|f_n-f_k\|_{H^1_L(X)}.
$$
Since $L^{1,\vc}(X)$ is complete, we can define $[b, T]f = \lim_{n\rightarrow \infty} [b, T]f_n$ in $L^{1,\vc}(X)$. It can be verified that
$$
\Big\| [b, T]f\Big\|_{L^{1,\vc}(X)}\leq C\|b\|_{BMO}\|f\|_{H^1_L(X)}.
$$

The proof of (ii) is complete.
\begin{flushright}
    $\Box$
\end{flushright}

\begin{rem}\label{rem0'}

(a) In our approach, to obtain (\ref{cond0}), we split
$$Ta=T(I-e^{-r_B^2L})^ma+ T(I-(I-e^{-r_B^2L})^m)a .$$
Observe that
$$ I-(I-e^{-r_B^2L})^m =
\sum_{k=1}^{m}c_ke^{-kr_B^2L},
$$
where $c_k=(-1)^{k+1}\f{m!}{(m-k)!k!}$. Therefore,
\begin{equation*}
T[I-(I-e^{-r_B^2L})^m]a =
\sum_{k=1}^{m}a_kT\Big(r_B^2L
e^{-\f{k}{m}r_B^2L}\Big)^m (r_B^{-2m}b),
\end{equation*}
where $a=L^m b$ and $a_k=c_k\Big(\f{k}{m}\Big)^m$.\\
Therefore, it is not difficult to see that condition (\ref{cond0}) holds if the following two estimates are satisfied:
\begin{equation}\label{cond1}
\Big(\int_{S_j(B)}|T(I-e^{-r_B^2L})^m f|^2\
d\mu\Big)^{\frac{1}{2}}\leq C2^{-2jm}\Big(\int_B|f|^2d\mu\Big)^{1/2}
\end{equation}
and
\begin{equation}\label{cond2}
\Big(\int_{S_j(B)}| T(r_B^{2m} L^me^{-k r_B^2L}) g  |^2\
d\mu\Big)^{\frac{1}{2}}\leq C2^{-2jm}\Big(\int_B|g|^2d\mu\Big)^{1/2}
\end{equation}
 for all integers  $j\geq 2$, $k = 1, \cdots , m$ and
 for all $f$ and $g$ with their supports contained in the ball $B$. \\

 (b) Theorem \ref{thm0}  still holds if the value $2^{-2jm}$ in (\ref{cond0}),
(\ref{cond1})  and (\ref{cond2})
is replaced by $2^{-2j\delta}$ for  some $\delta >\f{n(2-p)}{4p}$.\\

(c) The estimates in (\ref{cond1})  and (\ref{cond2}) do not hold without the
terms $(I-e^{-r_B^2L})^m$ and $r_B^{2m} L^me^{-k r_B^2L}$ in the left hand sides, respectively.
The effect of these terms is to make the kernels of  $T(I-e^{-r_B^2L})^m$
and $T(r_B^{2m} L^me^{-k r_B^2L})$ decay faster than the kernel of $T$ when $x$ is away from $y$.\\
 \end{rem}

\section{Commutators of BMO functions and the Riesz transforms or square functions on doubling manifolds}
Let $X$ be a complete non-compact connected Riemannian manifold,
$\mu$ the Riemannian measure, $\na$ the Riemannian gradient. Denote
by $|\cdot|$ the length in the tangent space, and by $||\cdot||_p$
the norm in $L^p(X, \mu), 1 \leq p \leq \vc$. For simplicity we will
write $L^p(X)$ instead of $L^p(X,\mu)$. Let $\lb$ be the Laplace-Beltrami operator.
Denote by $B(x, r)$ the open ball of radius $r > 0$ and center $x
\in X$, and by $V (x, r)$ its measure $\mu(B(x, r))$. Throughout this
section, assume that $X$ satisfies the doubling property (\ref{boubling}).
 It is well-known that the Laplace-Beltrami operator $\lb$ satisfies conditions $\bf{(H1)}$ and $\bf{(H2)}$.
 So let us denote  the Hardy space associated to $\lb$ by $H^1_\lb(X)$.\\

Let us consider $T = \Rie$, the Riesz transform on $X$, and take
$b \in$ BMO$(X)$ (the space of functions of bounded mean oscillations on
$X$). We define the commutator
$$
[b, T ]g = b T g - T (b g),
$$
where $g, b$ are scalar valued and $[b, T ]g$ is valued in the
tangent space. In \cite{AM}, it was proved that for any
$b\in BMO(X)$, under the doubling condition and Gaussian upper bound for the heat kernel, the commutator $[b,T]$ is bounded on $L^{p}(X)$ with appropriate weights, for $1<p< 2$. The case of end-point value $p=1$ was not considered in \cite{AM}.\\

Our following theorem gives the endpoint estimate for the commutator $[b, T]$ when $p=1$.
\begin{thm}\label{thm1}
Assume that $X$ satisfies the doubling property (\ref{boubling}) and $b$ is a function in
BMO$(X)$. Then, the Riesz transform $T=\Rie$ is bounded from $H^p_\lb(X)$ to $L^p(X)$, for all $0<p\leq 1$. Moreover, if the Riesz transform $T=\Rie$ is of weak type $(1,1)$ then the commutator $[b, T]$ maps $H^1_\lb(X)$ continuously into
$L^{1,\vc}(X)$.
\end{thm}
\begin{proof} By a similar argument to the proof of Lemma
2.2 in \cite{HMa}, it can be verified that for every $m\in \mathbb{N}$, all closed sets $E, F$ in $X$ with
$d(E,F)>0$ and every $f\in L^2(X)$ supported in $E$, one has
\begin{equation}\label{eq2}
||\na\lb^{-1/2}(I-\et)^mf||_{L^2(F)}\leq
C\Big(\f{t}{d(E,F)^2}\Big)^m||f||_{L^2(E)}, \ \forall t>0,
\end{equation}
and
\begin{equation}\label{eq3}
||\na\lb^{-1/2}(t\lb\et)^mf||_{L^2(F)}\leq
C\Big(\f{t}{d(E,F)^2}\Big)^m||f||_{L^2(E)}, \ \forall t>0.
\end{equation}
Obviously, (\ref{eq2}) and (\ref{eq3}) imply (\ref{cond1}) and (\ref{cond2}), respectively.
Hence our results follow from  Theorem \ref{thm0}. This completes our proof.
\end{proof}

Note that in (ii) of Theorem \ref{thm0} we
need  the Riesz transform $T=\Rie$ to be of weak type $(1,1)$ and
this condition on the Riesz transform can be obtained from
the assumptions of  doubling condition and the Gaussian bound
$\bf{(H3)}$. For reader's convenience, we recall the following
result in \cite{CD1}.\\
\begin{prop}[\cite{CD1}]\label{Duong'sthm}
Assume that $X$ satisfies the doubling property (\ref{boubling}) and the
kernels $p_t(x,y)$ of $e^{-t\lb}$ have Gaussian upper bounds $\bf{(H3)}$.
Then the Riesz transform $T=\Rie$ is bounded on $L^p(X), 1<p\leq 2$
and of weak type $(1,1)$.
\end{prop}
 From Theorems \ref{thm1} and Proposition \ref{Duong'sthm} we obtain
the following result.
\begin{cor}\label{cor2}
Assume that  $X$ satisfies the doubling property (\ref{boubling}), the kernels
$p_t(x,y)$  of $e^{-t\lb}$ have Gaussian upper bounds $\bf{(H3)}$ and
$b\in BMO(X)$. Then the commutator $[b, T]$ maps $H^1_\lb(X)$
continuously into $L^{1,\vc}(X)$.
\end{cor}
In \cite{HLMMY}, it was shown that  under condition
$\bf{(H3)}$, $H^1_{\lb,1}(X)=H^1_\lb(X)$ and $H^p_{\lb,m}(X)=L^p(X)$
for all $p>1$ and $m\geq 1$, see also \cite{AMR}. It implies that if $T$
is a bounded operator on $L^p(X), p>1$ and $T$ maps $H^1_\lb(X)$
continuously into $L^{1,\vc}(X)$ then by interpolation (Theorem 9.7
in \cite{HLMMY}), $T$ maps $H^q_{\lb,m}(X)$ into $L^q(X)$ whenever
$1<q<p$, and hence $T$ extends to a bounded operator on $L^q(X)$ for
all $1<q<p$.\\

\begin{defn} We say that $X$ satisfies an $L^2$ Poincar\'e inequality on balls if there exists $C>0$
such that for any ball $B\subset X$ and any function $f\in C^\infty(2B)$,
\begin{equation}\label{PoincareInequality}
\int_B|f(x)-f_B|^2 d\mu\leq Cr_B^2\int_{2B}|\nabla f(x)|^2d\mu,
\end{equation}
where $f_B$ denotes the mean-value of $f$ on the ball $B$ and $r_B$ the radius of $B$.
\end{defn}

Under the conditions of doubling property and $L^2$ Poincar\'e inequality, it is known that the Hardy space $H^1_\lb(X)$ coincides
with the  standard Hardy space. Indeed, we have the following result.

\begin{prop}[\cite{AMR}]
Assume that $X$ satisfies the doubling property (\ref{boubling}) and Poincar\'e inequality (\ref{PoincareInequality}) then the Hardy space $H^1_\lb(X)$ and the Coifman-Weiss Hardy space $H^1_{CW}(X)$ coincide.
\end{prop}
This Proposition and Corollary \ref{cor2} give the following Corollary.
\begin{cor}\label{thm3}
Assume that $X$ satisfies the doubling property (\ref{boubling}) and Poincar\'e inequality (\ref{PoincareInequality}) and let
$b\in BMO(X)$. Then the commutator $[b, T]$ maps $H^1_{CW}(X)$
continuously into $L^{1,\vc}(X)$.
\end{cor}

\smallskip

 We now show that similar results hold when we replace the Riesz transforms
of the Laplace-Beltrami operator by square functions of the Laplace-Beltrami operator.
 Consider the following four versions of the square functions
$$
\mathcal{G}f(x):=\Big(\int_{0}^\infty t|\na e^{-t\sqrt{\lb}}f(x)|^2dt\Big)^{1/2},\ \
\mathcal{H}f(x):=\Big(\int_{0}^\infty |\na e^{-t\lb}f(x)|^2dt\Big)^{1/2},
$$
$$
gf(x):=\Big(\int_{0}^\infty t|\sqrt{\lb} e^{-t\sqrt{\lb}}f(x)|^2dt\Big)^{1/2},
\ \
hf(x):=\Big(\int_{0}^\infty t |\lb e^{-t\lb}f(x)|^2dt\Big)^{1/2}.
$$
By similar arguments used in Lemma 2.2 of \cite{HMa}, it can be verified that $\mathcal{G}, \mathcal{H}, g$ and $h$ satisfy (\ref{eq2}) and (\ref{eq3}) and hence they satisfy (\ref{cond1}) and (\ref{cond2}). For reader's convenience, we sketch the proof for $\mathcal{H}$ only. The remainders are treated similarly.

\smallskip
The first ingredient is that the Davies-Gaffney estimate (\ref{D-Gestimate}) is valid in a general complete, connected Riemannian manifold
(see for example \cite{ACDH}): There exist two constants $C\geq 0$ and $c>0$ such that, for every $t\geq 0$,
every closed subsets $E$ and $F$ of $X$, and every function $f$ supported in $E$, one has
\begin{equation}\label{D-Gestimateonmanifold}
\begin{aligned}
||e^{-t\lb}f||_{L^2(F)}+||t\lb e^{-t\lb}f||_{L^2(F)}&+||\sqrt{t}|\na e^{-t\lb}f|||_{L^2(F)}\\
&\leq C\exp\Big\{-\f{d^2(E,F)}{t}\Big\}||f||_{L^2(E)}.
\end{aligned}
\end{equation}
Secondly, we recall the following result in \cite{HMa}.
\begin{lem}\label{lemD-Gestimateforproduct} Assume that the two families of operators $\{S_t\}_{t> 0}$ and $\{T_t\}_{t>
0}$ satisfy the Davies-Gaffney estimate (\ref{D-Gestimate}). Then there
exist two constants $C\geq 0$ and $c>0$ such that, for every $t>0$,
every closed subsets $E$ and $F$ of $X$, and every function $f$
supported in $E$, one has
$$
||S_sT_t f||_{L^2(F)}\leq
C\exp\Big\{-\frac{d(E,F)^2}{c\max\{s,t\}}\Big\}||f||_{L^2(E)}.
$$
\end{lem}
We now  show that $\mathcal{H}$ satisfies (\ref{eq2}) and (\ref{eq3}). Let us prove condition (\ref{eq2}) first. We have
\begin{equation*}
\begin{aligned}
||\mathcal{H}&(I-e^{-t\lb})^mf||_{L^2(F)}:=\Big\|\Big(\int_{0}^\infty |\na e^{-s\lb}(I-e^{-t\lb})^mf|^2ds\Big)^{1/2}\Big\|_{L^2(F)}\\
&\leq C\Big\|\Big(\int_{0}^\infty |\sqrt{s}\na e^{-s(m+1)\lb}(I-e^{-t\lb})^mf|^2\f{ds}{s}\Big)^{1/2}\Big\|_{L^2(F)}\\
&\leq C\Big(\int_{0}^t \Big\|\sqrt{s}\na e^{-s(m+1)\lb}(I-e^{-t\lb})^mf\Big\|_{L^2(F)}^2\f{ds}{s}\Big)^{1/2}\\
&~~~~+ C\Big(\int_{t}^\infty \Big\|\sqrt{s}\na e^{-s(m+1)\lb}(I-e^{-t\lb})^mf\Big\|_{L^2(F)}^2\f{ds}{s}\Big)^{1/2}=I_1+I_2.\\
\end{aligned}
\end{equation*}
To estimate the term $I_1$, we note that
$$
(I-e^{-t\lb})^m=I+\sum_{k=1}^mc_k e^{-tk\lb}.
$$
Hence,
\begin{equation*}
\begin{aligned}
I_1&\leq C\Big(\int_{0}^t \Big\|\sqrt{s}\na e^{-s(m+1)\lb}f\Big\|_{L^2(F)}^2\f{ds}{s}\Big)^{1/2}\\
&~~~~~+C\sup\limits_{1\leq k\leq m}\Big(\int_{0}^t \Big\|\sqrt{s}\na e^{-s(m+1)\lb}e^{-tk\lb}f\Big\|_{L^2(F)}^2\f{ds}{s}\Big)^{1/2}\\
&\leq C\Big(\int_{0}^t \Big\|\sqrt{s}\na e^{-s(m+1)\lb}f\Big\|_{L^2(F)}^2\f{ds}{s}\Big)^{1/2}\\
&~~~~~+C\sup\limits_{1\leq k\leq m}\Big(\int_{0}^t \Big\|\na e^{-s(m+1)\lb}\sqrt{tk}e^{-tk\lb}f\Big\|_{L^2(F)}^2\f{ds}{t}\Big)^{1/2}.
\end{aligned}
\end{equation*}
Using Lemma \ref{lemD-Gestimateforproduct} and (\ref{D-Gestimateonmanifold}), one has
\begin{equation*}
\begin{aligned}
I_1&\leq C\Big(\int_{0}^t \exp\Big\{-\f{d(E,F)^2}{cs}\Big\}\f{ds}{s}\Big)^{1/2}||f||_{L^2(E)}\\
&~~~ +C \Big( \exp\Big\{-\f{d(E,F)^2}{ct}\Big\}\int_{0}^t\f{ds}{t}\Big)^{1/2}||f||_{L^2(E)}.\\
\end{aligned}
\end{equation*}
It is not difficult to see that the expression above is bounded by $C\Big(\f{t}{d(E,F)}\Big)^m||f||_{L^2(E)}$ as desired.\\
We now estimate the second term $I_2$. We have
\begin{equation}\label{eq5.3}
\begin{aligned}
I_2\leq C\Big(\int_{t}^\infty \Big\|\sqrt{s}\na e^{-s\lb}(e^{-s\lb}-e^{-(s+t)\lb})^mf\Big\|_{L^2(F)}^2\f{ds}{s}\Big)^{1/2}.
\end{aligned}
\end{equation}
It was observed in \cite{HMa} that
\begin{equation*}
\begin{aligned}
\Big\|\f{s}{t}(e^{-s\lb}-e^{-(s+t)\lb})g\Big\|_{L^2(F)}^2\leq C\exp\Big\{-\f{d(E,F)^2}{cs}\Big\}||g||_{L^2(E)}.
\end{aligned}
\end{equation*}
Multiplying and dividing $(\ref{eq5.3})$ by $\Big(\f{s}{t}\Big)^{2m}$ and using Lemma \ref{lemD-Gestimateforproduct} for $\sqrt{s}\na e^{-s\lb}$ and $m$ copies of $\f{s}{t}(e^{-s\lb}-e^{-(s+t)\lb})$, we get that
$$
I_2\leq C \Big(\int_{t}^\infty \exp\Big\{-\f{d(E,F)^2}{cs}\Big\}\Big(\f{t}{s}\Big)^{2m}\f{ds}{s}\Big)^{1/2}||f||_{L^2(E)}.
$$
Next, making a change of variables $r:=\f{d(E,F)^2}{cs}$, we can control the RHS of the above expression by $C\Big(\f{t}{d(E,F)}\Big)^m||f||_{L^2(E)}$ as desired.\\
This finishes the proof of $(\ref{eq2})$ for $\mathcal{H}$. The proof for (\ref{eq3}) can be done essentially the same way, hence it is omitted here.
This completes our proof.\\

Let us recall that under the Gaussian condition $\bf{(H3)}$, $\mathcal{G}, \mathcal{H}, g$ and $h$ are of weak type $(1,1)$, see \cite{CD2}.
Hence the following result follows from Theorem \ref{thm0}.
\begin{thm}
\begin{enumerate}[(i)]
\item Assume that $X$ satisfies the doubling property (\ref{boubling}) and $b$ is a function in
BMO$(X)$. Then $\mathcal{G}, \mathcal{H}, g$ and $h$ are bounded from $H^p_\lb(X)$ to $L^p(X)$ for any $0<p\leq 1$.
\item If the Gaussian condition $\bf{(H3)}$ is satisfied, then the commutators of a BMO function $b$ and
each of $\mathcal{G}, \mathcal{H}, g$ and $h$ are bounded from $H_\lb^1(X)$ to $L^{1,\infty}(X)$.
\end{enumerate}
\end{thm}

\section{Commutators of BMO functions and Riesz transforms associated with  magnetic
Schr\"odinger operators}
The approach in Section 4.1 can be used to obtain the boundedness of Riesz transforms of Schr\"odinger operators and their commutators but  is not applicable in the case of magnetic Schr\"odinger operators. Indeed, a
different approach is needed for magnetic Schr\"odinger operators. \\

Consider magnetic Schr\"odinger operators in
general setting as in \cite{DOY}. Let the real vector potential
${\vec a}=(a_1, \cdots, a_n)$ satisfy
\begin{eqnarray}
a_k\in L^2_{\rm loc}({\mathbb R}^n), \ \ \ \ \forall k=1, \cdots, n,
\ \ \label{e1.1}
\end{eqnarray}
and an electric potential $V$ with
\begin{eqnarray}
0\leq V\in L^1_{\rm loc}({\mathbb R}^n). \label{e1.2}
\end{eqnarray}

\noindent Let $L_k={\partial/\partial x_k}-i a_k $. We   define the
form $Q$ by
\begin{eqnarray*}
Q(f,g)=\sum_{k=1}^n\int_{{\mathbb R}^n}L_kf
   {\overline {L_kg}} \  \! dx + \int_{{\mathbb R}^n}V f
   {\overline {g}} \ \! dx
  \end{eqnarray*}
  \noindent
with domain $\mathcal{D}(Q)=\mathcal{Q}\times \mathcal{Q} $ where
$$ \mathcal{Q}=\{f\in L^2({\mathbb R}^n), L_kf\in L^2({\mathbb R}^n) {\rm \  for}\
k=1,\cdots, n {\rm  \ and} \ \sqrt{V} f\in L^2({\mathbb R}^n)\}.
$$

 \noindent
 It is well known that this symmetric form is closed and this form coincides with the minimal closure
of the form given by the same expression but defined on
$C^{\infty}_0({\mathbb R}^n)$ (the space of $C^{\infty}$ functions
with compact supports). See, for example \cite{Si}.

Let us denote by $A$ the self-adjoint operator associated with $Q$.
The domain of $A$ is given by
$$
{\mathcal D}(A)=\Big\{f\in {\mathcal D}(Q), \exists g\in L^2({\Bbb
R}^n) {\rm \ such \ that\ } Q(f,\varphi)=\int_{{\mathbb R}^n} g{\bar
\varphi} dx, \ \ \forall\varphi\in {\mathcal D}(Q)\Big\},
$$
and $A$ is given by the expression
\begin{eqnarray}
Af=\sum_{k=1}^nL_k^{\ast} L_kf+Vf. \label{e1.3}
\end{eqnarray}

\noindent Formally, we write  $A=-(\nabla-i{\vec a})\cdot
(\nabla-i{\vec a})+V$. For  $k=1, \cdots, n$, the operators
$L_kA^{-1/2}$ are called the Riesz transforms associated with $A.$
It is easy to check that
\begin{eqnarray}
 \|L_kf \|_{L^2({\mathbb R}^n)}\leq  \|A^{1/2} f \|_{L^2({\mathbb R}^n)},
 \ \ \ \ \ \  \  \forall f\in {\mathcal
D}(Q)={\mathcal D}(A^{1/2}) \label{e1.4}
\end{eqnarray}

\noindent for any  $k=1, \cdots, n$, and hence the operators
$L_kA^{-1/2}$ are bounded on $L^2({\mathbb R}^n)$. Note that this is
also true for $V^{1/2}A^{-1/2}$. Moreover, it was recently proved in
Theorem 1.1 of  \cite{DOY}   that for each $k=1, \cdots, n$,  the
Riesz transforms $L_k A^{-1/2}$ and $V^{1/2} A^{-1/2}$ are   bounded
on $L^p({\mathbb R}^n)$ for all $1<p\leq 2$, i.e., there exists a
constant $C_p>0$ such that
\begin{eqnarray}
\label{ev}\hspace{1cm}
    \big\|V^{1/2} A^{-1/2}f\big\|_{L^p({\mathbb R}^n)}+
    \sum_{k=1}^n  \big\|L_k A^{-1/2}f\big\|_{L^p({\mathbb R}^n)} \leq C_p\|f\|_{L^p({\mathbb R}^n)}, \ \ \
\end{eqnarray}
for $1<p\leq 2$.

\noindent
The $L^p$-boundedness of Riesz transforms for the range $p>2$ can be obtained if one imposes certain additional regularity conditions on the potential $V$, see for example \cite{AB}.

\begin{rem}
\begin{enumerate}[(i)]

\item In \cite{DOY}, the boundedness of the Riesz transforms
$L_kA^{-1/2}$ and $V^{1/2}A^{-1/2}$ was proved for
$L^{p}(\mathbb{R}^n)$ spaces with $1<p<2$;
\item In \cite{DY1}, $L^p$
boundedness of commutators of a BMO function and the Riesz
transforms $L_kA^{-1/2}$ and $V^{1/2}A^{-1/2}$ was proved for the
range $1 < p < 2$;
\item Recently, \cite{A} extended the results in \cite{DOY} and \cite{DY1} to
weighted weak type $L^{1,\infty}$ estimates and weighted $L^p$
estimates with an appropriate range of $p$ (depending on the weight).
\end{enumerate}

\end{rem}

It is a natural open question to consider the endpoint estimates for the commutators of a BMO function and the operators
$L_kA^{-1/2}$ and $V^{1/2}A^{-1/2}$, and the boundedness of the Riesz transforms $L_kA^{-1/2}$ and $V^{1/2}A^{-1/2}$
on the range $0<p\leq 1$. Our aim in this section is to establish the end point estimates for the commutators of the
Riesz transforms $L_kA^{-1/2}$ and $V^{1/2}A^{-1/2}$ and a BMO function $b$ when $p=1$ and the
estimates for Riesz transforms $L_kA^{-1/2}$ and $V^{1/2}A^{-1/2}$ for $0<p\leq 1$. Our main result
of this section  is the following theorem.
\begin{thm}\label{thm4}

(i) The Riesz transforms $\LA$ and $\VA$ are bounded from $H^p_A(\mathbb{R}^n)$ to $L^p(\mathbb{R}^n)$ for all $0<p\leq 1$.

(ii) Let $b\in BMO(\mathbb{R}^n)$. Then the
commutators $\Big[b,\VA\Big]$ and $\Big[b,\LA\Big]$ map $H^1_A(\mathbb{R}^n)$
continuously into $L^{1,\infty}(\mathbb{R}^n)$.
\end{thm}

Recently, we had learnt that in \cite{JYY} the authors also obtained the results in (i) of Theorem \ref{thm4} by using the different approach.

\subsection{ Some kernel estimates on heat semigroups}
 \noindent
 Let $A=-(\nabla-i{\vec a})\cdot (\nabla-i{\vec a})+V$  be the magnetic Schr\"odinger
operator  in (\ref{e1.3}). By the well known diamagnetic inequality
(see, Theorem 2.3 of \cite{Si} and \cite{CFKS} for instance) we have
the pointwise inequality
\begin{eqnarray*}
\big|e^{-tA}f(x)\big|\leq e^{t\triangle}\big(|f|\big)(x) \ \ \
\forall t\geq 0, \ \ f\in L^2({\mathbb R}^n).
\end{eqnarray*}
\noindent This inequality implies in particular that the semigroup
$e^{-tA}$ maps $L^1({\Bbb R}^n)$ into $L^{\infty}({\Bbb R}^n)$ and
that the kernel $p_t(x,y)$ of $e^{-tA}$ satisfies
\begin{eqnarray}
  \big|p_t(x,y)\big|\leq  (4\pi t)^{-{n\over 2}}
 \exp\Big(-\frac{|x-y|^2}{4t}\Big)
\label{e2.5}
\end{eqnarray}

\noindent
  for all $t>0$ and  almost all $x,y\in {\mathbb R}^n$.

Note that $A$ satisfies conditions $\bf{(H1)}$ and $\bf{(H2)}$. So, for $0<p\leq 1$, we denote by $H^p_A(\mathbb{R}^n)$ the Hardy space associated to the operator $A$.
Note that Gaussian upper bounds carry over from heat kernels to
their time derivatives of its kernels. That is, for each
$k\in{\mathbb N}$, there exist two positive constants $c_k$ and
$C_k$ such that the time derivatives of $p_t$ satisfy
\begin{eqnarray}
\Big|{\partial_t^k}  p_t(x,y) \Big|\leq
  C_k   t^{-(n+2k)/2}\exp\Big(-\frac{|x-y|^2}{c_kt}\Big)
\label{e3.1}
\end{eqnarray}

\noindent for all $t>0$ and  almost all $x,y\in {\mathbb R}^n$. For
the proof of (\ref{e3.1}), see, for example, \cite{CD}, \cite{Da}
and \cite[Theorem 6.17]{Ou}.\\

We let $\widetilde{p}_t^k(x,y)=t^k(d^k/dt^k)p_t(x,y)$. In the sequel, we always use the notation $L_k \widetilde{p}_t^k(x,y)$ to mean $L_k \widetilde{p}_t^k(\cdot,y)(x)$.

\subsection{The proof of boundedness of commutators}
To prove the main result of this section,  we need the following lemma which gives a
weighted estimate for $L_k\widetilde{p}^k_t(x,y)$.
\begin{lem}\label{lem3}
Let $A=-(\nabla -i\vec{a})(\nabla -i\vec{a})+V$ be the magnetic
Schr\"odinger operator in (\ref{e1.3}). For each $k$ and $\gamma>0$, there exists $C>0$ such that
\begin{equation}\label{eq2.1}
\int_{\RR^n}|V^{1/2}(x)\widetilde{p}^k_t(x,y)|^2e^{\gamma \frac{|x-y|^2}{t}}\
dx+\sum_{k=1}^{n}\int_{\RR^n}|L_k\widetilde{p}^k_t(x,y)|^2 e^{\gamma
\frac{|x-y|^2}{t}}\ dx\leq \frac{C}{t^{\frac{n+2}{2}}}
\end{equation}
for all $t$ and $y\in \RR^n$.
\end{lem}
\begin{proof}
In \cite{DOY}, the authors proved this for $k=1$ and in the case when $k=0$, the proof can be found in \cite{A}.
We now adapt these estimates to prove (\ref{eq2.1}).\\
Let $\psi$ be a $C^\infty$ function with compact support on $\RR^n$
such that $0\leq\psi\leq 1$. Consider
$$
I_t(\psi)=\sum_{k=1}^n\int_{{\mathbb R}^n} |L_k \widetilde{p}^k_t(x,y)|^2
e^{\gamma\f{|x-y|^2}{t}}\psi(x) dx.
$$
\noindent Using Lemma 2.5 in \cite{Si}, we have
\begin{eqnarray}
I_t(\psi)&=& \sum_{k=1}^n\int_{{\mathbb R}^n} {\partial \over
\partial x_k}\Big(e^{-i\lambda_k}{ \widetilde{p}^k_t}(x,y)\Big) {\overline {
{\partial \over \partial x_k}\Big(e^{-i\lambda_k}{  \widetilde{p}^k_t}(x,y)\Big)}}
e^{\gamma\f{|x-y|^2}{t}}\psi(x) dx\nonumber\\[2pt]
&=& \sum_{k=1}^n\int_{{\mathbb R}^n} {\partial \over \partial
x_k}\Big(e^{-i\lambda_k}{  \widetilde{p}^k_t}(x,y)\Big) {\overline { {\partial
\over \partial x_k}\Big(e^{-i\lambda_k}{
p_t}(x,y)e^{\gamma\f{|x-y|^2}{t}}\psi(x)\Big)}}
  dx\nonumber\\[2pt]
&&- \sum_{k=1}^n\int_{{\mathbb R}^n} {\partial \over \partial
x_k}\Big(e^{-i\lambda_k}{  \widetilde{p}^k_t}(x,y)\Big) {\overline {
e^{-i\lambda_k}{  \widetilde{p}^k_t}(x,y)}} {\partial \over \partial
x_k}\Big(e^{\gamma\f{|x-y|^2}{t}}\psi(x)\Big)
  dx\nonumber\\[2pt]
&=&{  II_1 -II_2},
\label{e2.00}
\end{eqnarray}
where $\lambda_1, \ldots, \lambda_n$ are functions in  $L^2_{{\rm loc}}$ satisfying
$$
L_k=e^{i\lambda_k}\f{\partial}{\partial x_k}e^{-i\lambda_k}, \ \ k=1, \ldots, n.
$$

\noindent From the fact that $\psi$ has compact support, we have
$$
{\widetilde{p}^k_t}(\cdot,y)e^{\gamma\f{|\cdot-y|^2}{t}}\psi(\cdot)\in {\mathcal
D}(Q)\subset {\mathcal D}(L_k).
$$

\noindent We can then write the first term ${  II_1}$ as
$$
{  II_1}=\sum_{k=1}^n\int_{{\mathbb R}^n} L_k{\widetilde{p}^k_t}(x,y) {\overline {
L_k ({\widetilde{p}^k_t}(\cdot,y)e^{\gamma\f{|\cdot-y|^2}{t}}\psi)(x) }}
  dx.
$$

\noindent Since $0\leq V$ and $0\leq \psi$, we obtain
\begin{eqnarray*}
{  II_1}&\leq&Q({\widetilde{p}^k_t}(\cdot,y),\ {\widetilde{p}^k_t}(\cdot,y)e^{\gamma\f{|\cdot-y|^2}{t}}\psi)\\[2pt]
&=&\int_{{\mathbb R}^n} A{\widetilde{p}^k_t}(x,y) {\overline {  {\widetilde{p}^k_t}(x,y)}}
e^{\gamma\f{|x-y|^2}{t}}\psi(x) dx.
\end{eqnarray*}

\noindent On the other hand, we have $A{\widetilde{p}^k_t}(x,y) =t^{k}{d^{k+1}\over d  t^{k+1}}
p_t(x,y).$ We then apply (\ref{e3.1}) to obtain
$$
{  II_1} \leq
\int_{\RR^n}\f{1}{t^{n/2+1}}e^{-c_2\f{|x-y|^2}{t}}\frac{1}{t^{n/2}}e^{-c_1\f{|x-y|^2}{t}}e^{\gamma\f{|x-y|^2}{t}}\psi(x)dx.
$$
Hence for any $\gamma<c_1$ there exists a constant $c>0$ independent
of $\psi$ such that
\begin{equation}\label{e3.3}
{  II_1} \leq \f{c}{t^{n/2+1}},
\end{equation}
since $0\leq \psi\leq 1$.\\
Next, we rewrite the term ${  II_2}$   as follows:
\begin{eqnarray*}
{  II_2}&=& \sum_{k=1}^n\int_{{\mathbb R}^n} e^{i\lambda_k}
{\partial \over \partial x_k}\Big(e^{-i\lambda_k}{\widetilde{p}^k_t}(x,y)\Big)
{\overline
{p_t}(x,y)} e^{\gamma\f{|x-y|^2}{t}}\\[2pt]
&~&~~~\times\Big[{\partial \over \partial
x_k}\psi(x)+\f{2\beta(x_k-y_k)}{t}\psi(x)\Big]dx.
\end{eqnarray*}
This gives
\begin{eqnarray*}
{  II_2}&=& \sum_{k=1}^n\f{c}{\sqrt{t}}\int_{{\mathbb R}^n} |L_k{\widetilde{p}^k_t}(x,y)||{p_t}(x,y)| e^{2\gamma\f{|x-y|^2}{t}}\psi(x)dx\\[2pt]
&~&~~~+\sum_{k=1}^n\int_{{\mathbb R}^n} |L_k{\widetilde{p}^k_t}(x,y)||{\widetilde{p}^k_t}(x,y)|
e^{\gamma\f{|x-y|^2}{t}}\Big|{\partial \over \partial
x_k}\psi(x)\Big|dx\\[2pt]
&=& J_1(\psi)+J_2(\psi).
\end{eqnarray*}
 \noindent
Then by (\ref{e3.1}) and Cauchy-Schwarz inequality,
\begin{eqnarray*}
J_1(\psi)&\leq& \f{c}{\sqrt{t}}\sum_{k=1}^n\Big(\int_{{\mathbb R}^n} t^{-n}e^{3\beta-2c_1\f{|x-y|^2}{t}}dx\Big)^{\f{1}{2}}\Big(\int_{{\mathbb R}^n} |L_k{\widetilde{p}_t}(x,y)|^2e^{\gamma\f{|x-y|^2}{t}}\psi(x)dx\Big)^{\f{1}{2}}\\[2pt]
&\leq& \f{c}{\sqrt{t^{n/2+1}}}\sqrt{I_t(\psi)},
\end{eqnarray*}
provided $\beta<\f{2c_1}{3}$.\\
\noindent Using this estimate and (\ref{e3.3}), we have
\begin{equation}\label{e3.4}
I_t(\psi)\leq c\Big(\f{1}{t^{n/2+1}}+J_2(\psi)\Big),
\end{equation}
where $c$ is a constant independent of $\psi$.\\
Now apply (\ref{e3.4}) with $\psi_j(x)=\psi(x/j)$, where $\psi$ is a
function such that $\psi(x)=1$ for all $x$ with $|x|\leq 1$. It is
not difficult to show that $\lim\limits_{j\rightarrow\infty}J_2(\psi_j)=0$.
Then apply Fatou's lemma to (\ref{e3.4}) with $\psi_j$ we have
$$
\int_{{\mathbb R}^n}
|L_k{\widetilde{p}^k_t}(x,y)|^2e^{\gamma\f{|x-y|^2}{t}}dx\leq
\f{c}{t^{n/2+1}}.
$$
For the estimate of $\int_{{\mathbb R}^n}
|V^{1/2}(x)\widetilde{p}^k_t(x,y)|^2e^{\gamma\f{|x-y|^2}{t}}dx$, we note that
\begin{eqnarray*}
\int_{{\mathbb R}^n}
|V^{1/2}(x)\widetilde{p}^k_t(x,y)|^2e^{\gamma\f{|x-y|^2}{t}}\psi(x) dx =Q(\widetilde{p}^k_t(\cdot, y), \widetilde{p}^k_t(\cdot, y)e^{\gamma\f{|\cdot-y|^2}{t}}\psi)-
II_1.
\end{eqnarray*}
From the estimates of both terms, we obtain that
$$
\int_{{\mathbb R}^n}
|V^{1/2}(x)\widetilde{p}^k_t(x,y)|^2e^{\gamma\f{|x-y|^2}{t}}\psi(x) dx \leq \f{c}{t^{n/2+1}}.
$$
At this stage, repeating the above argument, one has
$$
\int_{{\mathbb R}^n}
|V^{1/2}(x)\widetilde{p}^k_t(x,y)|^2e^{\gamma\f{|x-y|^2}{t}} dx \leq \f{c}{t^{n/2+1}}.
$$
This finishes our proof.
\end{proof}

The following proposition gives an estimate for the operators $\LA$
and $\VA$ which will be useful for the proof of Theorem \ref{thm4}
\begin{prop}\label{prop2.1} For all $m\geq 1$ there exists $C>0$ such that
for all $j\geq 2$, all balls $B$ and all $f\in L^1(\RR^n)$ with
support in $B$
\begin{equation}\label{eq4.1}
\Big(\int_{S_j(B)}|\VA(I-e^{-r_B^2A})^m f(x)|^2\
dx\Big)^{\frac{1}{2}}\leq C\ 2^{-2j(m-1)}\Big(\int_B|f(x)|^2 dx\Big)^{1/2}
\end{equation}
and for all $k = 1, \cdots, n$,
\begin{equation}\label{eq4.2}
\Big(\int_{S_j(B)}|\LA(I-e^{-r_B^2A})^m f(x)|^2\
dx\Big)^{\frac{1}{2}}\leq C\ 2^{-2j(m-1)}\Big(\int_B|f|^2dx\Big)^{1/2} .
\end{equation}
\end{prop}
\begin{proof} We will prove only (\ref{eq4.2}). The inequality (\ref{eq4.1}) can be treated by
a similar argument.\\
We  adapt an argument used in \cite{ACDH} (see also \cite{A}) to our situation. Fix a ball $B$ with radius $r_B$ and $f\in L^1(\RR^n)$ supported in
$B$. Observe that
\begin{equation*}
A^{-1/2}=\frac{1}{\sqrt{\pi}}\int_{0}^{\infty}e^{-tA}\frac{dt}{\sqrt{t}} .
\end{equation*}
We have
\begin{equation*}
\LA(I-e^{-r_B^2A})^mf=c\int_{0}^{\infty}L_ke^{-tA}(I-e^{-r_B^2A})^m\frac{dt}{\sqrt{t}}=c\int_{0}^{\infty}g_r(t)L_ke^{-tA}dt,
\end{equation*}
where $g_r:\RR^+\rightarrow \RR$ is a function such that
\begin{equation*}
\int_{0}^{\infty}|g_r(t)|e^{\frac{-c4^jr_B^2}{t}}\frac{dt}{\sqrt{t}}
\leq C_m4^{-jm}.
\end{equation*}
See [ACDH, p.932].
Hence the composite operator $\LA(I-e^{-r_B^2A})^m$ has an associated kernel
$\mathcal{K}_{s,k}(y,z)$ defined by
\begin{equation*}
\mathcal{K}_{s,k}(y,z)=c\int_{0}^{\infty}g_r(t)L_kp_s(y,z)dt.
\end{equation*}
By invoking
Lemma \ref{lem3},
\begin{equation*}
 \Big(\int_{\RR^n}|L_kp_t(x,y)|^2e^{\gamma
\frac{|x-y|^2}{t}}\ dx\Big)^{\frac{1}{2}}\leq
\frac{C}{t^{\frac{n+2}{4}}}
\end{equation*}
for all $t>0$ and $y\in \RR^n$ and some $\gamma>0$. This implies
that for all $j>0, y\in B$ and all $t>0$,
\begin{equation*}
\begin{aligned}
 \Big(\int_{S_j(B)}|L_kp_t(x,y)|^2 \ dx\Big)^{\frac{1}{2}}
&\leq \frac{C}{\sqrt{t}}e^{\frac{-c4^jr_B^2}{t}}\frac{|2^{j}B|^{1/2}}{t^{\frac{n}{4}}}\frac{1}{|2^{j}B|^{1/2}}\\
&\leq \frac{C}{\sqrt{t}}e^{\frac{-c4^jr_B^2}{t}}\Big(\frac{4^jr_B^2}{t}\Big)^{\frac{n}{4}}\frac{1}{|2^{j}B|^{1/2}}\\
&\leq \frac{C}{\sqrt{t}}e^{\frac{-\alpha4^jr_B^2}{t}}\frac{1}{|2^{j}B|^{1/2}}\\
&\leq \frac{C}{\sqrt{t}}e^{\frac{-\alpha4^jr_B^2}{t}}\frac{1}{|B|^{1/2}}\\
\end{aligned}
\end{equation*}
for some $\alpha<c$. Note that in the last inequality we use the fact that
$s^be^{-cs}<Ce^{-\alpha s}$ for all $s>0$ and $\alpha<c$.\\
Using Minkowski's inequality we obtain that the LHS of (\ref{eq4.2})
is dominated by \begin{equation*}
\begin{aligned}
\int_{0}^{\infty}|g_r(t)|\int|f(y)&|\Big(\int_{S_j(B)}|L_kp_s(x,y)|^2dx\Big)^{1/2}dydt\\
&\leq C \frac{1}{|B|^{1/2}}\int_{0}^{\infty}|g_r(t)|e^{\frac{-\alpha
4^jr_B^2}{t}}\frac{dt}{\sqrt{t}}\Big(\int_B|f(y)|dy\Big)\\
&\leq C
4^{-mj}\f{1}{|B|^{1/2}}\int_B|f(y)|dy\\
&\leq C
4^{-mj}\Big(\int_B|f(x)|^2dx\Big)^{1/2}.
\end{aligned}
\end{equation*}
The proof is complete.
\end{proof}

\begin{prop}\label{prop2.2} For all $m\geq 1$ and $t\geq 0$ there exists $C>0$ such that
for all $j\geq 2$, all balls $B$, $\f{1}{m}r_B^2\leq t\leq r_B^2$ and all $f$ with
support in $B$
\begin{equation}\label{eq5.1}
\Big(\int_{S_j(B)}|\VA(tAe^{-tA})^m f(x)|^2\
dx\Big)^{\frac{1}{2}}\leq C2^{-2jm}\Big(\int_B|f(x)|^2dx\Big)^{1/2}
\end{equation}
and for all $k = 1, \cdots , n$,
\begin{equation}\label{eq5.2}
\Big(\int_{S_j(B)}|\LA (tAe^{-tA})^m f(x)|^2\
dx\Big)^{\frac{1}{2}}\leq C2^{-2jm}\Big(\int_B|f(x)|^2dx\Big)^{1/2} .
\end{equation}
\end{prop}
\begin{proof}
Firstly, observe that
$$\LA f=\f{1}{2\sqrt{\pi}}\int_{0}^\infty L_k e^{-sA}f\f{ds}{\sqrt{s}}=\f{m}{2\sqrt{\pi}}\int_{0}^\infty L_k e^{-msA}f\f{ds}{\sqrt{s}}.
$$
So, we obtain
\begin{equation*}
\begin{aligned}
\Big(\int_{S_j(B)}|&\LA (tAe^{-tA})^m f(x)|^2\
dx\Big)^{\frac{1}{2}}\\
&= C\Big(\int_{S_j(B)}\Big|\int_0^\infty L_k e^{-msA}(tAe^{-tA})^mf(y)dy\f{ds}{\sqrt{s}}\Big|^2\
dx\Big)^{\frac{1}{2}}\\
&= C\Big(\int_{S_j(B)}\Big|\int_0^\infty\int_{B}\Big(\f{t}{t+s}\Big)^{m} L_k \widetilde{p}^m_{m(t+s)}(x,y)f(y)dy\f{ds}{\sqrt{s}}\Big|^2\
dx\Big)^{\frac{1}{2}}\\
\end{aligned}
\end{equation*}
By Minkowski's inequality and Lemma \ref{lem3},
\begin{equation*}
\begin{aligned}
\Big(\int_{S_j(B)}|&\LA (tAe^{-tA})^m f(x)|^2\
dx\Big)^{\frac{1}{2}}\\
&\leq C \int_0^\infty\Big(\f{t}{t+s}\Big)^{m}\int_{B}f(y) \Big(\int_{S_j(B)}\Big|L_k \widetilde{p}^m_{m(t+s)}(x,y)\Big|^2\
dx\Big)^{\frac{1}{2}}dy\f{ds}{\sqrt{s}}\\
&\leq C \int_0^\infty \Big(\f{t}{t+s}\Big)^{m}\exp\Big\{-c\f{d^2(S_j(B),B)}{t+s}\Big\}\f{ds}{\sqrt{s}(t+s)^{\f{n+2}{4}}}\int_{B}f(y)dy\\
&\leq C \int_0^\infty \Big(\f{t}{t+s}\Big)^{m}\exp\Big\{-c\f{4^jr_B^2}{t+s}\Big\}\f{ds}{\sqrt{s}(t+s)^{\f{n+2}{4}}}r_B^{n/2}||f||_{L^2(B)}\\
&\leq C ||f||_{L^2(B)}\Big(\int_0^{t}\ldots + \int_{t}^\infty\ldots\Big)=C||f||_{L^2(B)}(I+II).
\end{aligned}
\end{equation*}
Let us estimate $I$ first. We have
\begin{equation*}
\begin{aligned}
\int_0^{t} \Big(\f{t}{t+s}\Big)^{m}&\exp\Big\{-c\f{4^jr_B^2}{t+s}\Big\} \Big(\f{r_B^2}{t+s}\Big)^{\f{n+2}{4}}\f{ds}{r_B\sqrt{s}}\\
&\leq C\int_0^{t} \Big(\f{t}{t+s}\Big)^m\Big(\f{t+s}{4^jr_B^2}\Big)^{m}\Big(\f{r_B^2}{t}\Big)^{\f{n+2}{4}}\f{ds}{r_B\sqrt{s}}\\
&\leq C\int_0^{t} \Big(\f{t}{4^jr_B^2}\Big)^m\f{ds}{r_B\sqrt{s}}\\
&\leq C 4^{-jm}.
\end{aligned}
\end{equation*}
We now estimate the second term $II$. We have
\begin{equation*}
\begin{aligned}
\int_{t}^\infty \Big(\f{t}{t+s}\Big)^m&\exp\Big\{-c\f{4^jr_B^2}{t+s}\Big\} \Big(\f{r_B^2}{t+s}\Big)^{\f{n+2}{4}}\f{ds}{r_B\sqrt{s}}\\
&\leq C\int_t^{\infty} \Big(\f{t}{t+s}\Big)^m\Big(\f{t+s}{4^jr_B^2}\Big)^{m-1}\Big(\f{r_B^2}{t}\Big)^{\f{n+2}{4}}\f{ds}{r_B\sqrt{s}}\\
&\leq C4^{-j(m-1)}\int_t^{\infty} \Big(\f{t}{r_B^2}\Big)^m\f{r_Bds}{s\sqrt{s}}\\
&\leq C 4^{-j(m-1)}.
\end{aligned}
\end{equation*}
This completes our proof.
\end{proof}

We are now ready to give the proof of Theorem \ref{thm4}.

\medskip

\emph{Proof of Theorem \ref{thm4}:}
Firstly, it can be proved that the Riesz transforms $\LA$ and $\VA$ are of weak type $(1,1)$, see [DOY, p. 273]. So, combining Propositions \ref{prop2.1} and \ref{prop2.2} together with Theorem \ref{thm0} and Remark \ref{rem0'} (b), Theorem \ref{thm4} is proved.
\begin{flushright}
    $\Box$
\end{flushright}

\section{Holomorphic functional calculi and spectral multipliers}

 \subsection{Preliminaries on holomorphic functional calculus}
 We now give some preliminary definitions of holomorphic functional
calculi as introduced by A. McIntosh \cite{Mc}.\\

Let $0\leq \omega<\nu<\pi$. We define the closed sector in the
complex plane $\mathbb{C}$
$$
S_\omega=\{z\in \mathbb{C}: |\arg z|\leq \omega\}
$$
and denote the interior of $S_\omega$ by $S^0_\omega $.
We employ the following subspaces of the space $H(S^0_\nu)$ of all
holomorphic functions on $S_\nu^0$:
$$
H_\vc(S_\nu^0)=\{g \in H(S^0_\nu): ||g||_\vc<\vc\},
$$
where $||g||_{\vc}=\sup\{|g(z)|: z\in S_\nu^0\}$, and
$$
\Psi(S^0_\nu)=\{\psi\in H(S^0_\nu): \exists s>0, |\psi(z)|\leq
c|z|^s(1+|z|^{2s+1})^{-1}\}.
$$
Let $0\leq\omega<\pi$. A closed operator $L$ in $L^2(X)$ is said to
be of type $\omega$ if $\si(L) \subset S_\omega$, and for each
$\nu>\omega$ there exists a constant $c_\nu$ such that
$$
||(L-\lambda I)^{-1}||\leq c_\nu|\lambda|^{-1}, \lambda\notin S_\nu.
$$
If $L$ is of type $\omega$ and  $\psi\in \Psi(S^0_\nu)$, we define
$\psi(L) \in \mathcal{L}(L^2, L^2)$ by
$$
\psi(L)=\frac{1}{2\pi i}\int_{\Gamma}(L-\lambda
I)^{-1}\psi(\lambda)d\lambda,
$$
where $\Gamma$ is the contour $\{\xi=re^{\pm i\xi}: r>0\}$
parametrized clockwise around $S_\omega$, and $\omega<\xi<\nu$.
Clearly, this integral is absolutely convergent in $\mathcal{L}(L^2,
L^2)$, and it is straightforward to show, using Cauchy's theorem,
that the definition is independent of the choice of $\xi\in
(\omega,\nu)$. If, in addition, $L$ is one-one and has dense range
and if $\phi \in H_\vc(S^0_\nu)$, then $\phi(L)$ can be defined by $$
\phi(L)=[\psi(L)]^{-1}(\phi\psi)(L),
$$
where $\psi(z)=z(1+z)^{-2}$.

It can be shown that $\phi(L)$ is a
well-defined linear operator in $L^2$. We say that $L$ has a bounded
$H_\vc$ calculus in $L^2$ if there exists $c_{\nu,2} > 0$ such that
$\phi(L) \in \mathcal{L}(L^2, L^2)$, and for $\phi \in H_\vc(S^0_\nu)$,
$$
||\phi(L)||\leq c_{\nu,2}||\phi||_\vc.
$$
In \cite{Mc} it was proved that $L$ has a bounded $H_\vc$-calculus
in $L^2$ if and only if for any non-zero function $\psi\in
\Psi(S^0_\nu)$, $L$ satisfies the square function estimate and its
reverse
\begin{equation}\label{squareestimate}
c_1||f||_2\leq \Big(\int_0^\vc
||\psi_t(L) f ||_2^2\f{dt}{t}\Big)^{1/2}\leq c_2|| f ||_2
\end{equation}
for some $0 < c_1 \leq  c_2 < \vc$, where  $\psi_t(x) =  \psi(tx)$.
Note that different choices of $\nu>\omega$ and $\psi\in
\Psi(S^0_\nu)$ lead to equivalent quadratic norms of $f$. As noted
in \cite{Mc}, positive self-adjoint operators satisfy the quadratic
estimate (\ref{squareestimate}). So do normal operators with spectra
in a sector, and maximal accretive operators. For definitions of
these classes of operators, we refer the reader to \cite{Yo}. For
detailed studies of operators which have bounded holomorphic functional
calculi, see for example \cite{AHS, Mc, Ha, DR}.

\subsection{Application to holomorphic functional calculus}
We first show that the holomorphic functional calculi $g(L)$ satisfy (\ref{cond1}) and (\ref{cond2}).
\begin{prop}\label{holoprop}
Assume that $L$ satisfies $\bf{(H1)}$ and $\bf{(H2)}$. Let $0<\nu<\pi$. Then for any $g\in H_\vc(S_\nu^0)$ $m\in \mathbb{N}$, all closed sets $E, F$ in $X$ with
$d(E,F)>0$ and any $f\in L^2(X)$ supported in $E$, one has
\begin{equation}\label{ho5.1}
\begin{aligned}
||g(L)&(I-e^{-tL})^mf||_{L^2(F)}\\
&\leq
C\max\Big\{\Big(\f{t}{d(E,F)^2}\Big)^{m-1},\Big(\f{t}{d(E,F)^2}\Big)^{m+1}\Big\}||f||_{L^2(E)}||g||_{\infty}, \ \forall t>0,
\end{aligned}
\end{equation}
and
\begin{equation}\label{ho5.2}
\begin{aligned}
||g(L)&(tLe^{-tL})^mf||_{L^2(F)}\\
&\leq
C\max\Big\{\Big(\f{t}{d(E,F)^2}\Big)^{m-1},\Big(\f{t}{d(E,F)^2}\Big)^{m+1}\Big\}||f||_{L^2(E)}||g||_{\infty}, \ \forall t>0.
\end{aligned}
\end{equation}
\end{prop}
Before giving the proof of Proposition \ref{holoprop}, we state the following lemma.
\begin{lem}\label{holem}
Assume that $L$ satisfies conditions $\bf{(H1)}$ and $\bf{(H2)}$. Then for any $z=re^{i\theta}$ with $\theta \in (-\pi/2,\pi/2)$, all closed sets $E, F$ in $X$ with $d(E,F)>0$ and any $f\in L^2(X)$ supported in $E$, one has
$$
\Big\|e^{-zL}f\Big\|_{L^2(F)}\leq C \exp\Big\{-\f{d(E,F)^2}{|z|}\cos\theta \Big\}||f||_{L^2(E)}.
$$
\end{lem}
The proof is similar to the proof of Proposition 3.1 of \cite{HLMMY} and hence we omit details here.\\
\emph{Proof of Proposition \ref{holoprop}:} We prove the estimate (\ref{ho5.1}) first. We define
$$g_{s,t}(z)=z^s(1+z)^{-2s}g(z)(1-e^{-tz})^m .$$
 Then $g_{s,t}\in \Psi(S^0_\nu)$. Moreover, $\lim\limits_{s\rightarrow 0}g_{s,t}(z)=g(z)(1-e^{-tz})^m$ uniformly in any compact set in $S^0_\nu$. Therefore, by the Convergence Theorem in \cite{Mc},
\begin{equation}\label{convergenceofHFC}
\lim\limits_{s\rightarrow 0}g_{s,t}(L)f=g(L)(1-e^{-tL})^mf
\end{equation}
in $L^2(X)$ norm for every $f\in L^2(X)$.\\
Choose $0<\theta<\mu<\nu$ and $\mu<\f{\pi}{2}$. On one hand, we have
$$
g_{s,t}(L) = \f{1}{2\pi i}\int_{\gamma}(L-\lambda I)^{-1}\f{\lambda^s}{(1+\lambda)^{2s}}g(\lambda)(1-e^{-t\lambda})^md\lambda,
$$
here $\gamma = \gamma_+ + \gamma_-$, where $\gamma_+(t)=te^{i\mu}$ if $0\leq t<\infty$ and $\gamma_-(t)=-te^{-i\mu}$ if $0\leq t<\infty$ with $0<\mu<\nu$. On the other hand, we have
$$
(L-\lambda I)^{-1}= \int_{\Gamma}e^{\lambda z}e^{-zL}dz,
$$
here $\Gamma = \Gamma_+ + \Gamma_-$, where $\Gamma_+(t)=te^{i\beta}$ if $0\leq t<\infty$ and $\gamma_-(t)=-te^{-i\beta}$ if $0\leq t<\infty$ with $\beta = \f{\pi - \theta}{2}$.
Therefore
\begin{equation*}
\begin{aligned}
||g_{s,t}(L)f||_{L^2(F)}\leq &\Big\|\int_{\Gamma_+}e^{-zL}f\int_{\gamma_+}\f{\lambda^s}{(1+\lambda)^{2s}}g(\lambda)e^{\lambda z}(1-e^{-t\lambda})^md\lambda dz\Big\|_{L^2(F)}\\
&~~+ \Big\|\int_{\Gamma_-}e^{-zL}f\int_{\gamma_-}\f{\lambda^s}{(1+\lambda)^{-2s}}g(\lambda)e^{\lambda z}(1-e^{-t\lambda})^md\lambda dz\Big\|_{L^2(F)}\\
&=I_1+I_2.
\end{aligned}
\end{equation*}
Let us estimate $I_1$ first. Observe that, by Lemma \ref{holem},
\begin{equation*}
\begin{aligned}
I_1&\leq C  \int_{\Gamma_+}\Big\|e^{-zL}f\Big\|_{L^2(F)}\int_{\gamma_+}\f{\lambda^s}{(1+\lambda)^{2s}}g(\lambda)e^{\lambda z}(1-e^{-t\lambda})^md\lambda dz\\
&\leq C \|f\|_{L^2(E)}\int_{0}^{\infty}\exp\Big\{-c\f{d(E,F)^2}{|z|}\cos\beta\Big\}\int_{0}^{\infty}||g(\lambda)||_{\infty}|e^{\lambda z}(1-e^{-t\lambda})^m|d|\lambda| d|z|.\\
\end{aligned}
\end{equation*}
Since $\mu<\pi/2$, we obtain $|(1-e^{-t\lambda})^m|\leq c (t|\lambda|)^m$. Hence,
\begin{equation*}
\begin{aligned}
I_1&\leq C  \|f\|_{L^2(E)}||g(\lambda)||_{\infty}\int_{0}^{\infty}\exp\Big\{-c\f{d(E,F)^2}{|z|}\Big\}\int_{0}^{\infty}e^{-c_1|\lambda| |z|}(t|\lambda|)^md|\lambda| d|z|\\
&\leq C  \|f\|_{L^2(E)}||g(\lambda)||_{\infty}\int_{0}^{\infty}\exp\Big\{-c\f{d(E,F)^2}{|z|}\Big\}\f{t^m}{|z|^{m+1}} d|z|\\
&\leq C  \|f\|_{L^2(E)}||g(\lambda)||_{\infty}\Big(\int_{0}^{t}\ldots + \int_{t}^{\infty}\ldots\Big)=C  \|f\|_{L^2(E)}||g(\lambda)||_{\infty}\Big(I_{11} + I_{12}\Big).\\
\end{aligned}
\end{equation*}
Concerning the term $I_{11}$, we have
\begin{equation*}
\begin{aligned}
I_{11}&=\int_{0}^{t}\exp\Big\{-c\f{d(E,F)^2}{|z|}\Big\}\f{t^m}{|z|^{m+1}} d|z|\\
&\leq \int_{0}^{t} \Big(\f{|z|}{d(E,F)^2}\Big)^{m+1}\f{t^m}{|z|^{m+1}} d|z|=\Big(\f{t}{d(E,F)^2}\Big)^{m+1}.
\end{aligned}
\end{equation*}
Concerning the term $I_{11}$, we have
\begin{equation*}
\begin{aligned}
I_{12}&=\int_{t}^{\infty}\exp\Big\{-c\f{d(E,F)^2}{|z|}\Big\}\f{t^m}{|z|^{m+1}} d|z|\\
&\leq \int_{t}^{\infty} \Big(\f{|z|}{d(E,F)^2}\Big)^{m-1}\f{t^m}{|z|^{m+1}} d|z|=\Big(\f{t}{d(E,F)^2}\Big)^{m-1}.
\end{aligned}
\end{equation*}
The term $I_2$ can be treated by the same way. These estimates together with (\ref{convergenceofHFC}) give (\ref{ho5.1}).\\
We proceed to prove (\ref{ho5.2}). Repeating the arguments above, we obtain
\begin{equation*}
\begin{aligned}
||g(L)(tLe^{-tL})^mf||_{L^2(F)}\leq &\Big\|\int_{\Gamma_+}e^{-zL}f\int_{\gamma_+}g(\lambda)e^{\lambda z}(t\lambda e^{-t\lambda })^md\lambda dz\Big\|_{L^2(F)}\\
&~~+ \Big\|\int_{\Gamma_-}e^{-zL}f\int_{\gamma_-}g(\lambda)e^{\lambda z}(t\lambda e^{-t\lambda })^md\lambda dz\Big\|_{L^2(F)}\\
&=II_1+II_2.
\end{aligned}
\end{equation*}
We need only estimate $II_1$. The estimate for   $II_2$ is proved similarly. One has,
\begin{equation*}
\begin{aligned}
II_1&\leq C  \int_{\Gamma_+}\Big\|e^{-zL}f\Big\|_{L^2(F)}\int_{\gamma_+}g(\lambda)e^{\lambda z}(t\lambda e^{-t\lambda })^m d\lambda dz\\
&\leq C \|f\|_{L^2(E)}\int_{0}^{\infty}\exp\Big\{-c\f{d(E,F)^2}{|z|}\cos\beta\Big\}\int_{0}^{\infty}||g(\lambda)||_{\infty}|e^{\lambda z}(t\lambda e^{-t\lambda })^m|d|\lambda| d|z|\\
&\leq C \|f\|_{L^2(E)}\int_{0}^{\infty}\exp\Big\{-c\f{d(E,F)^2}{|z|}\Big\}\int_{0}^{\infty}||g(\lambda)||_{\infty}e^{-|\lambda| (c_1|z|+c_2t)}(t|\lambda| )^md|\lambda| d|z|\\
&\leq C \|f\|_{L^2(E)}||g(\lambda)||_{\infty}\int_{0}^{\infty}\exp\Big\{-c\f{d(E,F)^2}{|z|}\Big\}\f{t^m}{(c_1|z|+c_2t)^{m+1}} d|z|\\
&\leq C \|f\|_{L^2(E)}||g(\lambda)||_{\infty}\Big(\int_{0}^{t}\ldots + \int_{t}^{\infty}\ldots\Big)=C \|f\|_{L^2(E)}||g(\lambda)||_{\infty}\Big(II_{11} + II_{12}\Big)\\
\end{aligned}
\end{equation*}
For the term $II_{11}$,
\begin{equation*}
\begin{aligned}
II_{11}&=\int_{0}^{t}\exp\Big\{-c\f{d(E,F)^2}{|z|}\Big\}\f{t^m}{(c_1|z|+c_2t)^{m+1}} d|z|\\
&\leq C\int_{0}^{t}\Big(\f{|z|}{d(E,F)^2}\Big)^m\f{t^m}{|z|^{m+1}} d|z|=C\Big(\f{t}{d(E,F)^2}\Big)^{m+1}.\\
\end{aligned}
\end{equation*}
The remaining term $II_{12}$ is dominated by
\begin{equation*}
\begin{aligned}
C\int_{0}^{\infty}\Big(\f{|z|}{d(E,F)^2}\Big)^{m-1}\f{t^m}{z^{m+1}} d|z|=C\Big(\f{t}{d(E,F)^2}\Big)^{m-1}.\\
\end{aligned}
\end{equation*}
This completes our proof.
\begin{flushright}
    $\Box$
\end{flushright}

Note that from Proposition \ref{holoprop}, $g(L)$ satisfies (\ref{cond1}) and (\ref{cond2}) with $m$ being
replaced by $m -1$. Moreover, if the Gaussian upper bound condition $\bf{(H3)}$ is satisfied  then $g(L)$ is of weak type $(1,1)$, see \cite{DM}. Hence we obtain the following result.

\begin{thm}\label{thm HFC}
Assume that $L$ satisfies conditions $\bf{(H1)}$ and $\bf{(H2)}$. Let $g \in H_\vc(S_\nu^0)$. Then $g(L)$ is bounded from $H^p_L(X)$ to $L^p(X)$ for all $0<p\leq 1$. Moreover, if the Gaussian upper bound condition $\bf{(H3)}$ is satisfied, then the commutator of a BMO function $b$ and $g(L)$ is bounded from $H^1_L(X)$ to $L^{1,\infty}(X)$.
\end{thm}
\begin{rem}
We can obtain  the following estimate which is sharper than (\ref{ho5.2}):
\begin{equation}\label{ho5.3}
\begin{aligned}
||g(L)&(tLe^{-tL})^mf||_{L^2(F)}\leq C\Big(\f{t}{d(E,F)^2}\Big)^{m}||f||_{L^2(E)}||g||_{\infty}, \ \forall t>0.
\end{aligned}
\end{equation}
See for example \cite{AL}. We remark that the estimate (\ref{ho5.3}) implies the boundedness
for the holomorphic functional calculus $g(L)$ from $H^p_L(X)$ to $H^p_L(X)$ and hence $g(L)$ is bounded from
$H^p_L(X)$ to $L^p(X)$, see \cite{AL, DL}.  In this section, we obtain the $H^p_L-L^p$ boundedness of $g(L)$ and
our main result is the endpoint estimate of the commutator $[b, g(L)]$ where $b$ is a BMO function.
\end{rem}

\subsection{Application to spectral multipliers}
Assume that $L$ satisfies conditions $\bf{(H1)}$. Let
$$
F(L) =\int_0^\infty F(\lambda) dE_L(\lambda)
$$
be the spectral multiplier $F(L)$ defined on $L^2$ by using the spectral resolution of $L$.\\
Our main result on spectral multipliers is  the following.
\begin{prop}\label{specmultiplierpro}
Assume that $L$ satisfies conditions $\bf{(H1)}$ and $\bf{(H2)}$.
Let $F$ be a bounded function defined on $(0,\infty)$ such that for some real number $\alpha>\f{n(2-p)}{2p}+\f{1}{2}$ and any non-zero function $\eta\in C_c^\infty(\f{1}{2},2)$ there exists a constant $C_\eta$ such that
\begin{equation}\label{specequa}
\sup_{t>0}\|\eta(\cdot)F(t\cdot)\|_{W^{2,\alpha}(\mathbb{R})}\leq C_\eta
\end{equation}
where $\|F\|_{W^{q,\alpha}(\mathbb{R})}=\|(I-d^2/dx^2)^{\alpha/2}F\|_{L^q}$.
Then the multiplier operator satisfies the following estimate
\begin{equation}\label{eq1 spectral}
\Big(\int_{S_j(B)}  |F(\sqrt{L}) a|^2
d\mu\Big)^{\frac{1}{2}}\leq C2^{-j\delta}V(B)^{\f{1}{2}-\f{1}{p}}
\end{equation}
for some $\delta>\f{n(2-p)}{4p}$, for any $(p,2,m)$-atom $a$ supported in $B$ and sufficiently large $m$.
\end{prop}
Before giving the proof we state the following result in \cite{DP}.
\begin{lem}\label{spmultlem}
Let $\gamma>1/2$ and $\beta>0$. Then there exists a constant $C>0$ such that for every function $F\in W^{2,\gamma+\beta/2}$ and every function $g\in L^2(X)$ supported in the ball $B$, we have
$$
\int_{d(x,x_B)>2r_B}|F(2^j\sqrt{L})g(x)|^2\Big(\f{d(x,x_B)}{r_B}\Big)^{\beta}d\mu(x)\leq C(r_B2^j)^{-\beta}\|F\|_{W^{2,\gamma+\beta/2}}^2\|g\|_{L^2}^2
$$
for $j\in \mathbb{Z}$.
\end{lem}
\emph{Proof of Proposition \ref{specmultiplierpro}:}

Obviously, since $F(\sqrt{L})$ is bounded on $L^2(X)$, (\ref{eq1 spectral}) holds for $j=0, 1, 2 $. For $j>2$, we will exploit some ideas in \cite{DP} to our situation.\\
Fix $\epsilon>0$ and $\gamma>1/2$ such that $\gamma+\epsilon+\f{n(2-p)}{2p}=\alpha$. Set $\beta= \f{n(2-p)}{p}+2\epsilon$. Then $\gamma +\beta/2=\alpha$. Let $a=L^{m}b$ be a $(p,2,m)$-atom supported in $B$ with $m>\beta/4$ and $\ell_0=-\log_2r_B$.\\
Fix a function $\phi \in C_c^\infty(\f{1}{2},2)$ such that
$$
\sum_{\ell\in \mathbb{Z}}\phi(2^{-\ell}\lambda)=1 \ \text{for} \
\lambda>0.
$$
Then, one has
\begin{equation*}
\begin{aligned}
F(\sqrt{L})a=\sum_{\ell\geq
\ell_0}\phi(2^{-\ell}\sqrt{L})F(\sqrt{L})a+\sum_{\ell<
\ell_0}\phi(2^{-\ell}\sqrt{L})L^m F(\sqrt{L})b.
\end{aligned}
\end{equation*}
Recall that by definition of $(p,2,m)$ atoms,
$$
\|a\|_{L^2}\leq V(B)^{\f{1}{2}-\f{1}{p}} \ \text{and} \
\|b\|_{L^2}\leq r_B^{2m}V(B)^{\f{1}{2}-\f{1}{p}}.
$$
Set
$$
F_\ell(\lambda)=
\begin{cases}
F(2^\ell\lambda)\phi(\lambda) \ &, \ell\geq \ell_0\\
2^{2m\ell} F(2^\ell\lambda)\lambda^{m}\phi(\lambda) \ &, \ell< \ell_0.
\end{cases}
$$
and extend $F_\ell$ to the even function. Obviously,
$$
\|F_\ell\|_{W^{2,\alpha}}\leq
\begin{cases}
C \ &, \ell\geq \ell_0\\
C2^{2m\ell} \ &, \ell< \ell_0.
\end{cases}
$$
Applying Lemma \ref{spmultlem}, we obtain
\begin{equation*}
\begin{aligned}
\Big(\int_{S_j(B)}&|F(\sqrt{L})a(x)\Big(\f{d(x,x_B)}{r_B}\Big)^{\beta}d\mu(x)\Big)^{1/2}\\
& \leq C\sum_{\ell\geq \ell_0}(r_B2^\ell)^{-\beta/2}\|F_\ell\|_{W^{2,\alpha}}\|a\|_{L^2}+C \sum_{\ell< \ell_0}(r_B2^\ell)^{-\beta/2}\|F_\ell\|_{W^{2,\alpha}}\|b\|_{L^2}\\
& \leq C [\sum_{\ell\geq \ell_0}(r_B2^\ell)^{-\beta/2}+\sum_{\ell< \ell_0}(r_B2^\ell)^{2m-\beta/2}]V(B)^{\f{1}{2}-\f{1}{p}}\\
& \leq C V(B)^{\f{1}{2}-\f{1}{p}}.
\end{aligned}
\end{equation*}
This implies that
$$
\Big(\int_{S_j(B)}|F(\sqrt{L})a|^2 d\mu\Big)^{1/2}\leq
C2^{-j\beta/2}V(B)^{\f{1}{2}-\f{1}{p}}.
$$
The proof is complete.
\begin{flushright}
    $\Box$
\end{flushright}

From Proposition \ref{specmultiplierpro} and Theorem \ref{thm0} we obtain the following result.
\begin{thm}\label{spectralthm1}
\begin{enumerate}[(i)]
\item Assume that $L$ satisfies conditions $\bf{(H1)}$ and $\bf{(H2)}$.
Let $F$ be a bounded function defined on $(0,\infty)$ such that for some real number $\alpha>\f{n(2-p)}{2p}+\f{1}{2}$ and any non-zero function $\eta\in C_c^\infty(\f{1}{2},2)$ satisfies the condition (\ref{specequa}).
Then the multiplier operator $F(L)$ is bounded from $H^p_L(X)$ to $L^p(X)$ for $0<p<1$.

\item Under the same assumptions as $(i)$, the operators $F(L)$ is bounded from  $H^p_L(X)$ to
$H^p_L(X)$ for all $0<p\leq 1$.

\item Assume that $L$ satisfies $\bf{(H1)}$ and $\bf{(H3)}$. Let $F$ be a bounded function defined on $(0,\infty)$ such that for some real number $\alpha>\f{n}{2}+\f{1}{2}$ and any non-zero function $\eta\in C_c^\infty(\f{1}{2},2)$ there exists a constant $C_\eta$ such that
\begin{equation}\label{specequa1}
\sup_{t>0}\|\eta(\cdot)F(t\cdot)\|_{W^{2,\alpha}(\mathbb{R})}\leq C_\eta.
\end{equation}
Then the commutator of $F(L)$ and a BMO function $b$ is bounded from $H^1_L(X)$ to $L^{1,\infty}(X)$.
\end{enumerate}
\end{thm}
\begin{proof} (i) is a direct consequence of Proposition \ref{specmultiplierpro} and Theorem \ref{thm0}.

\smallskip

To prove (iii), we note that the Gaussian upper bound condition $\bf{(H3)}$ together with (\ref{specequa1}) implies that
 $F(L)$ is of weak type $(1,1)$, see \cite{DOS}. Therefore, by Proposition \ref{specmultiplierpro} and
 Theorem \ref{thm0}, the commutator of $F(L)$ and a BMO function $b$ is bounded from $H^1_L(X)$ to $L^{1,\infty}(X)$.

\smallskip
Concerning  (ii), we can prove (ii)    by first
showing  that $F(L)$ maps a $(1,2,m)$ atom to $H^1_L(X)$. It then follows that $F(L)$ is bounded on $H^p_L(X)$ for all $0<p\leq 1$. We sketch the proof here  for reader's convenience.

\smallskip

Due to Proposition \ref{mol-pro} and the fact that the condition
(\ref{specequa}) is invariant under the change of variable $\lambda
\mapsto \lambda^s$, it suffices to show that there exists
$\epsilon>0$ such that for any $(p,2,2m)$-atom $a$ associated to the ball $B$, the
function
$$
\widetilde{a}=F(\sqrt{L})a
$$
is a multiple of $(p,2,m,\epsilon)$-molecule.\\
Fix $\epsilon>0$ and $\gamma>1/2$ such that $\gamma+\epsilon+\f{n(2-p)}{2p}=\alpha$. Set $\beta= \f{n(2-p)}{p}+2\epsilon$. Then $\gamma +\beta/2=\alpha$. Let $j_0=-\log_2r_B$, $a=L^{2m}b$ with $m>\beta/4$ and $\widetilde{b}=F(\sqrt{L})L^mb$. Then $\widetilde{a}= L^m\widetilde{b}$. We need to verify that
\begin{equation}\label{eq0-spectral}
\|(r_B^2L)^\ell \widetilde{b}\|_{L^2(S_k(B))}\leq C2^{-k\epsilon}r_B^{2m}V(2^kB)^{1/2-1/p}
\end{equation}
for all $0\leq \ell\leq m$ and $k=0,1,\ldots$.

\smallskip
It is easy to check that (\ref{eq0-spectral}) holds for $k=0,1,2$. To check (\ref{eq0-spectral}) for $k\geq 3$, we fix a function $\phi \in C_c^\infty(\f{1}{2},2)$ such that
$$
\sum_{j\in \mathbb{Z}}\phi(2^{-j}\lambda)=1 \ \text{for} \
\lambda>0.
$$
Then, for $0\leq \ell\leq m$, one has
\begin{equation*}
\begin{aligned}
(r_B^2L)^\ell\widetilde{b}=&r_B^{2\ell}\sum_{j\geq j_0}\phi(2^{-j}\sqrt{L})F(\sqrt{L})L^{\ell+m}b\\
&+r_B^{2\ell}\sum_{j< j_0}\phi(2^{-j}\sqrt{L})L^m F(\sqrt{L})L^{\ell}b\\
&=r_B^{2\ell}\sum_{j\geq j_0}\phi(2^{-j}\sqrt{L})
F(\sqrt{L})b_1+r_B^{2\ell}\sum_{j< j_0}\phi(2^{-j}\sqrt{L})L^m
F(\sqrt{L})b_2.
\end{aligned}
\end{equation*}
It is easy to see that
$$
\|b_1\|_{L^2}\leq r_B^{2m-2\ell}V(B)^{\f{1}{2}-\f{1}{p}} \ \text{and} \
\|b_2\|_{L^2}\leq r_B^{4m-2\ell}V(B)^{\f{1}{2}-\f{1}{p}}.
$$
Set
$$
F_j(\lambda)=
\begin{cases}
F(2^j\lambda)\phi(\lambda) \ &, j\geq j_0\\
2^{2mj} F(2^j\lambda)\lambda^{2m}\phi(\lambda) \ &, j< j_0.
\end{cases}
$$
and extend $F_j$ to the even function. Obviously,
$$
\|F_j\|_{W^{2,\alpha}}\leq
\begin{cases}
C \ &, j\geq j_0\\
C2^{2mj} \ &, j< j_0.
\end{cases}
$$
Applying Lemma \ref{spmultlem}, we obtain
\begin{equation*}
\begin{aligned}
\Big(\int_{d(x,x_B)>2r_B} & |(r_B^2L)^\ell \widetilde{b}(x)|^2\Big(\f{d(x,x_B)}{r_B}\Big)^{\beta}d\mu(x)\Big)^{1/2}\\
& \leq Cr_B^{2\ell}\sum_{j\geq j_0}(r_B2^j)^{-\beta/2}\|F_j\|_{W^{2,\alpha}}\|b_1\|_{L^2}
+Cr_B^{2\ell}\sum_{j< j_0}(r_B2^j)^{-\beta/2}\|F_j\|_{W^{2,\alpha}}\|b_2\|_{L^2}\\
& \leq C r_B^{2m}V(B)^{\f{1}{2}-\f{1}{p}}.
\end{aligned}
\end{equation*}

So, for $k\geq 3$, one has
\begin{equation*}
\begin{aligned}
\|(r_B^2L)^\ell\widetilde{b}\|_{L^2(S_k(B))}
& \leq C 2^{-j(\beta/2 -n(\f{1}{p}-\f{1}{2}))}r_B^{2m}V(2^kB)^{\f{1}{2}-\f{1}{p}}\\
& \leq C 2^{-j\epsilon}r_B^{2m}V(2^kB)^{\f{1}{2}-\f{1}{p}}\\
\end{aligned}
\end{equation*}
This implies that $\widetilde{a}=F(\sqrt{L})a$ is a multiple of
a $(p,2,m,\epsilon)$-molecule and the multiple constant is independent
of $a$. Our proof is complete.
\end{proof}

\begin{rem}

i) When $p=1$,  it was shown in \cite{DP} that $F(L)$ maps a $(1,2,m)$ atom into $H^1_L(X)$,
but the boundedness of $F(L)$ on the Hardy spaces was only obtained under the extra assumption that
the measure of any ball $B(x,r)$ has a lower bound $c r^{\kappa}$ for some $\kappa > 0$, see \cite{DP}.
In this article, using the fact that  the convergence in the atomic decomposition in
Hardy spaces $H^p_L(X)$ is in the sense of $L^2(X)$, we can obtain the boundedness of $F(L)$ on $H^p_L(X)$
without the assumption of the measure of a ball having the lower bound $c r^{\kappa}$.

\medskip

ii) The condition (\ref{specequa1}) in (iii) of Theorem \ref{spectralthm1} can be replaced by the following condition
$$\sup_{t>0}\|\eta(\cdot)F(t\cdot)\|_{W^{\vc,\alpha}(\mathbb{R})}\leq C_\eta$$ with $\alpha>n/2$. In this situation, we can use the following estimates in \cite[Lemma 4.3]{DOS} instead of Lemma \ref{spmultlem} to obtain Proposition \ref{specmultiplierpro}. Since the proof is quite similar to that of Proposition \ref{specmultiplierpro}, we omit details here.

\begin{lem}\label{lem1}
Suppose that $L$ satisfies $\bf{(H1)}$ and $\bf{(H3)}$,
 $R>0$ and $\alpha>0$. Then for any $\epsilon>0$, there exists a
constant $C=C(\alpha, \epsilon)$ such that
\begin{eqnarray}\label{e4.2}
\int_X \big|K_{F(\sqrt{L})}(x,y)\big|^2 \big(1+Rd(x,y)\big)^{\alpha}
d\mu(x)\leq \f{C}{ V(y, R^{-1})}
 \|\delta_{R} F\|^2_{W^{\vc,\f{\alpha}{2} +\epsilon}}
\end{eqnarray}
for all Borel functions $F$ such that {\rm supp}$F\subseteq [R/4,
R]$, where $K_{F(\sqrt{L})}(x,y)$ is the associated kernel to $F(\sqrt{L})$.
\end{lem}

iii) The approach in this paper can be applied to consider the boundedness of the commutators of generalized fractional integrals on Hardy spaces associated to operators. This will appear in the forthcoming paper \cite{A2}.
\end{rem}

\textbf{Acknowledgement.} The second author was supported by ARC (Australian Research Council). The first author would like to thank Luong Dang Ky for his useful discussion. The authors would like to thank the referee for his/her useful comments and suggestions to improve the paper.


\begin{thebibliography}{19}


\bibitem{AHS} H. Amann, M. Hieber and G. Simonett, Bounded $H_{\infty}$-calculus for elliptic operators,
\emph{Differential and Integral Equations} \textbf{7} (1994),
 613-653.



\bibitem{AB} P. Auscher, and B. Ben Ali,  Maximal inequalities and Riesz transform
 estimates on $L^p$ spaces for Schr\"odinger operators with
 nonnegative potentials, \emph{Ann. Inst. Fourier (Grenoble)} \textbf{57} (2007), 1975-2013.

\bibitem{ACDH}
 P. Auscher, T. Coulhon, X. T. Duong, and S. Hofmann, Riesz
 transform on manifolds and heat kernel regularity, \emph{Ann. Sci. \'Ecole Norm. Sup.} \textbf{37} (2004), 911-957.

\bibitem{ADM} P. Auscher, X. T. Duong and A. McIntosh,
Boundedness of Banach
space valued singular integral operators and Hardy spaces,
unpublished manuscript.




\bibitem{AM} P. Auscher, and J. M. Martell, Weighted norm inequalities, off-diagonal estimates and elliptic operators.
 Part IV: Riesz transforms on manifolds and weights, \emph{Math. Z.} \textbf{260} (2008), 527-539.

\bibitem{AMR} P. Auscher, A. McIntosh, and E. Russ, Hardy spaces of differential forms on Riemannian manifolds.
\emph{J. Geom. Anal.} \textbf{18} (2008), 192-248.

\bibitem{AR} P. Auscher and E. Russ, Hardy spaces and divergence
operators on strongly Lipschitz domain of ${\mathbb R}^n$,
\emph{J. Funct. Anal.} \textbf{201} (2003), 148-184.

\bibitem{A} T. A. Bui, Weighted norm inequalities for Riesz transforms of magnetic
Schr\"odinger operators, \emph{Differential Integral Equations} \textbf{23} (2010), 811-826.

\bibitem{A2} T. A. Bui, The endpoint estimates of the commutators of generalized fractional integrals on Hardy spaces associated to operators, in preparation.

\bibitem{AL} T. A. Bui and J. Li, Orlicz-Hardy spaces associated to
operators satisfying bounded $H_\vc$ functional calculus and
Davies-Gaffney estimates, \emph{J. Math. Anal. Appl.} \textbf{373}
(2011), 485-501.


\bibitem{BK} S. Blunck and P. Kunstmann, Calder\'on-Zygmund theory for non-integral operators and the
H1-functional calculus, \emph{Rev. Mat. Iberoamericana} \textbf{19} (2003), 919-942

\bibitem{CD} T. Coulhon, and X. T. Duong, Maximal regularity
and kernel bounds: observations on a theorem by Hieber and
Pr\"uss, \emph{Adv. Differential Equations}  \textbf{5} (2000),
343-368.

\bibitem{CD1}  T. Coulhon and X. T. Duong, Riesz transforms for $1\leq p\leq
2$, \emph{Trans. Amer. Math. Soc.} \textbf{351} (1999),  1151-1169.

\bibitem{CD2} T. Coulhon, X. T. Duong, and X. D. Li, Littlewood-Paley-Stein functions on complete Riemannian manifolds
for $1\leq p\leq 2$, \emph{Studia Math.} \textbf{154} (2003), 37-57.

\bibitem{CFKS} H.L. Cycon, R.G. Foese, W. Kirsh, and B. Simon,
{\it Schr\"odinger operators with applications to quantum mechanics
and global geometry}, Texts and Monographs in Physics, Springer
Verlag, 1987.


\bibitem{CRW} R.R. Coifman, R. Rochberg, and G. Weiss,  Factorization
theorem for Hardy spaces in several variables, \emph{Ann. of Math.} \textbf{103} (1976), 611-635.

\bibitem{CW1}  R. Coifman and G. Weiss, Extensions of Hardy spaces and their use on analysis, \emph{Bull. Amer. Math. Soc.} \textbf{83} (1977),
569-645.

\bibitem{Da} E.B. Davies,  Non-Gaussian aspects of heat kernel behaviour,
  \emph{J. London Math. Soc.} \textbf{55} (1997), 105-125.

\bibitem{DL} X. T. Duong and J. Li, Hardy spaces associated to
operators satisfying bounded $H_\infty$ functional calculus and
Davies-Gaffney estimates, preprint 2009.

\bibitem{DM} X. T. Duong and A. McIntosh, Singular integral operators with non-smooth kernels on
irregular domains, \emph{Rev. Mat. Iberoamericana} \textbf{15} (1999), 233-265.

\bibitem{DR} X. T. Duong and D. Robinson, Semigroup kernels, Poisson bounds, and holomorphic
functional calculus,  \emph{J. Funct. Anal.} \textbf{142} (1996), 89-128.


\bibitem{DP} J. Dziuba\'nski and M. Preisner, Remarks on spectral multiplier theorems on Hardy spaces associated with semigroups of operators,
\emph{Revista de la uni\'on Matem\'atica Argentina} \textbf{50} (2009), 201-215.

\bibitem{DOS}  X. T. Duong, E. M. Ouhabaz  and A. Sikora,
 Plancherel-type estimates and sharp spectral multipliers,
  \emph{J. Funct. Anal.} \textbf{196} (2002),  443-485.



\bibitem{DOY} X.T. Duong, E.M. Ouhabaz, and L.X. Yan,  Endpoint estimates
for   Riesz transforms of magnetic Schr\"odinger operators, \emph{Ark. Mat.} \textbf{44} (2006), 261-275.

\bibitem{DY1} X. T. Duong, and L. Yan,  Commutators of Riesz transforms of magnetic Schr\"odinger operators,
\emph{Manuscripta Math.} \textbf{127} (2008), 19-234.

\bibitem{DY2} X. T. Duong and L. Yan, Duality of Hardy and BMO spaces
associated with operators with heat kernel bounds, \emph{J. Amer. Math. Soc.}
\textbf{18} (2005), 943-973.



\bibitem{G} A. Grigor\'yan, Gaussian upper bounds for the heat
kernel on arbitrary manifolds, \emph{J. Differential Geom.} \textbf{45} (1997), 33-52.

\bibitem{Ha} M. Haase, \emph{The Functional Calculus for Sectorial Operators}, Operator Theory:
Advances and Applications, vol. 169, Birkh\"auser Verlag, Basel, 2006.

\bibitem{HMa} S. Hofmann and J.M. Martell, $L^p$ bounds for Riesz transforms and square roots
associated to second order elliptic operators, \emph{Publ. Mat.} \textbf{47}
(2003), 497-515.

\bibitem{HM} S. Hofmann and S. Mayboroda, Hardy and BMO spaces associated to divergence
form elliptic operators, \emph{Math. Ann.} \textbf{344} (2009), 37-116.

\bibitem{HLMMY}S. Hofmann, G. Lu, D. Mitrea, M. Mitrea and L. Yan, Hardy spaces associated to non-negative self-adjoint
operators satisfying Davies-Gaffney estimates, \emph{Mem. Amer. Math. Soc.} \textbf{214} (2011).

\bibitem{HMMc} S. Hofmann, S. Mayboroda and A. McIntosh,
Second order elliptic operators with complex bounded measurable
coefficients in $L^p$, Sobolev and Hardy spaces, \emph{Ann. Sci. \'Ecole Norm. Sup.} \textbf{44} (2011), 723-800.

\bibitem{JYY} R. Jiang, Da. Yang and Do. Yang, Maximal function characterizations of Hardy spaces associated with magnetic Schr\"odinger operators, \emph{Forum Math.} \textbf{24} (2012), 471-494.

\bibitem{Yo} T. Kato, \emph{Perturbation Theory for Linear Operators}, Reprint of the corr.
print. of the 2nd ed. 1980. Classics in Mathematics. Berlin: Springer-Verlag.
xxi, 619 p. , 1995.

\bibitem{Mc}A. McIntosh, Operators which have an $H_\infty$-calculus, Miniconference on
operator theory and partial differential equations, \emph{Proc. Centre
Math. Analysis, ANU, Canberra} \textbf{14} (1986), 210-231.

\bibitem{Ou} E.M. Ouhabaz, {\it Analysis of heat equations on domains,} London
Math. Soc. Monographs, Vol. 31, Princeton Univ. Press 2004.

\bibitem{Pe} C. P\'erez, Endpoint estimates for commutators of singular integral operators, \emph{J. Funct.
Anal.} \textbf{128} (1995), 163-185.

\bibitem{S} A. Sikora, Riesz transform, Gaussian bounds and the method of wave function, \emph{Math. Z.} \textbf{247} (2004), 643-662.

\bibitem{Si} B. Simon, Maximal and minimal Schr\"odinger forms,
\emph{J. Op. Theory} \textbf{1} (1979), 37-47.

\bibitem{St} E.M. Stein,  {\it Harmonic analysis: Real variable
methods, orthogonality and oscillatory integrals}, Princeton Univ.
Press, Princeton, NJ, (1993).


\bibitem{Y} L.X. Yan,  Classes of Hardy spaces associated with operators,
duality theorem and applications, \emph{Trans. Amer. Math. Soc.}  \textbf{360} (2008), 4383-4408.




\end{thebibliography}
\end{document}